\documentclass[a4paper,11pt]{amsart}

\usepackage{latexsym}
\usepackage{a4wide}
\usepackage{amscd}
\usepackage{graphics}
\usepackage{amsmath}
\usepackage{amssymb}
\usepackage{mathrsfs}
\usepackage{enumerate}
\usepackage{hyperref}
\usepackage{pgf,tikz}

\newcommand{\R}{\mathbb R}
\newcommand{\C}{\mathbb C}
\newcommand{\N}{\mathbb N}
\newcommand{\Z}{\mathbb Z}

\newcommand{\abs}[1]{\left\vert #1 \right\vert}

\newcommand{\Di}{\mathcal{D}^{1,2}}
\renewcommand{\Re}{\mathop{\mathfrak{Re}}}

\newtheorem{Theorem}{Theorem}[section]
\newtheorem{Corollary}[Theorem]{Corollary}
\newtheorem{Lemma}[Theorem]{Lemma}
\newtheorem{Proposition}[Theorem]{Proposition}
 \theoremstyle{definition}

 \newtheorem{remark}[Theorem]{Remark}

\begin{document}

\title[Eigenvalue variation for Aharonov-Bohm moving pole]{On the
  leading term of the eigenvalue variation for Aharonov-Bohm operators
  with a moving pole}

\author{Laura Abatangelo, Veronica Felli}

\address{
\hbox{\parbox{5.7in}{\medskip\noindent
  L. Abatangelo, V. Felli\\
Dipartimento di Matematica e Applicazioni,\\
 Universit\`a di Milano Bicocca, \\
Via Cozzi 55, 20125 Milano (Italy). \\[2pt]
         {\em{E-mail addresses: }}{\tt laura.abatangelo@unimib.it, veronica.felli@unimib.it.}}}
}

\date{May 20, 2015}

\thanks{ 
The authors have been partially supported by 2014 INdAM-GNAMPA research project ``Stabilit\`{a}
spettrale e analisi asintotica per problemi singolarmente
perturbati'', 
2015 INdAM-GNAMPA research project ``Operatori di Schr\"odinger con potenziali elettromagnetici singolari: stabilit\`{a} spettrale e stime di decadimento'',
and by ERC Advanced Grant 2013 
``Complex Patterns for Strongly Interacting Dynamical Systems - COMPAT''.
}

\begin{abstract}
 We study the behavior of eigenvalues for magnetic
  Aharonov-Bohm operators with half-integer circulation and Dirichlet
  boundary conditions in a planar domain. 
We analyse the leading term in the Taylor
expansion of the eigenvalue function as the pole moves in the interior of
  the domain, proving that it is a harmonic homogeneous polynomial and detecting its exact coefficients.
\end{abstract}

\maketitle

\section{Introduction}

Completing the analysis performed in \cite{AF}, we deepen the
investigation of the behavior of eigenvalues for the magnetic
Aharonov-Bohm operator with half-integer circulation
 and Dirichlet
boundary conditions in a planar domain.  We refer to \cite{AF},
\cite{BNNNT},  \cite{NNT}, and \cite{NT} 
for motivations and references to previous related
literature.

For every $a=(a_1,a_2)\in\R^2$, the Aharonov-Bohm vector potential
with pole $a$ and circulation $1/2$ is defined as 
\[
A_a(x_1,x_2)=\frac12\bigg(\frac{-(x_2-a_2)}{(x_1-a_1)^2+(x_2-a_2)^2},
\frac{x_1-a_1}{(x_1-a_1)^2+(x_2-a_2)^2}\bigg),\quad 
(x_1,x_2)\in\R^2\setminus\{a\}.
\]
Let $\Omega\subset\R^2$ be a bounded, open and simply connected
domain containing the origin. For every $a\in\Omega$, we introduce the functional space
$H^{1 ,a}(\Omega,\C)$ as the completion of
\[
\{u\in
H^1(\Omega,\C)\cap C^\infty(\Omega,\C):u\text{ vanishes in a
  neighborhood of }a\}
\]
 with respect to the norm 
 $$
 \|u\|_{H^{1,a}(\Omega,\C)}=\left(\left\|\nabla u\right\|^2
   _{L^2(\Omega,\C^2)} +\|u\|^2_{L^2(\Omega,\C)}+\big\|\tfrac{u}{|x-a|}
   \big\|^2_{L^2(\Omega,\C)}\right)^{\!\!1/2},
$$
which, in view of the
Hardy type inequality proved in \cite{LW99} (see also \cite[Lemma 3.1
and Remark 3.2]{FFT}), is 
equivalent to the norm 
\begin{equation*}
  \left(\left\|(i\nabla+A_{a}) u\right\|^2
    _{L^2(\Omega,\C^2)} +\|u\|^2_{L^2(\Omega,\C)}\right)^{\!\!1/2}.
\end{equation*}
We denote as $H^{1 ,a}_{0}(\Omega,\C)$ the space obtained as the completion
of $C^\infty_{\rm c}(\Omega\setminus\{a\},\C)$ with respect to the
norm $\|\cdot\|_{H^{1}_{a}(\Omega,\C)}$.

For every $a\in\Omega$,  we say that $\lambda\in\R$ is an eigenvalue
of the 
problem 
\begin{equation}\label{eq:eige_equation_a}\tag{$E_a$}
  \begin{cases}
   (i\nabla + A_{a})^2 u = \lambda u,  &\text{in }\Omega,\\
   u = 0, &\text{on }\partial \Omega,
 \end{cases}
\end{equation}
in a weak sense if there exists $u\in
H^{1,a}_{0}(\Omega,\C)\setminus\{0\}$ (called eigenfunction) such that
\[
\int_\Omega (i\nabla u+A_{a} u)\cdot \overline{(i\nabla v+A_{a}
  v)}\,dx=\lambda\int_\Omega u\overline{ v}\,dx \quad\text{for all }v\in H^{1,a}_{0}(\Omega,\C).
\]
From classical spectral theory, the eigenvalue problem $(E_a)$ admits
a sequence of real diverging eigenvalues $\{\lambda_k^a\}_{k\geq 1}$
with finite multiplicity; in the enumeration
$\lambda_1^a \leq \lambda_2^a\leq\dots\leq \lambda_j^a\leq\dots$
each eigenvalue is repeated as many times as its multiplicity.  We are
interested in the  behavior of the function $a\mapsto \lambda_j^a$ in a
neighborhood of a fixed point  $b\in \Omega$; without loss of
generality, we can consider $b=0\in\Omega$. 

Let us assume that there exists $n_0\geq 1$ such that 
\begin{equation}\label{eq:1}
  \lambda_{n_0}^0\quad\text{is simple},
\end{equation}
and denote
\[
\lambda_0= \lambda_{n_0}^0
\]
and, for any $a\in\Omega$,
\[
 \lambda_a= \lambda_{n_0}^a.
 \]
 In \cite[Theorem 1.3]{lena} it is proved
that, for all $j\geq1$ such that assumption \eqref{eq:1} holds true, 
the function $a\mapsto \lambda_j^a$ is analytic in a neighborhood of
$0$. In particular $a\mapsto \lambda_j^a$ is continuous and, if $a\to 0$, then
\begin{equation}\label{eq:conv_auto}
  \lambda_a\to \lambda_0.
\end{equation}
Let $\varphi_0\in H^{1,0}_{0}(\Omega,\C)\setminus\{0\}$ be a
$L^2(\Omega,\C)$-normalized
eigenfunction of problem $(E_0)$ associated to the eigenvalue
$\lambda_0= \lambda_{n_0}^0$, i.e. satisfying
\begin{equation}\label{eq:equation_lambda0}
 \begin{cases}
   (i\nabla + A_0)^2 \varphi_0 = \lambda_0 \varphi_0,  &\text{in }\Omega,\\
   \varphi_0 = 0, &\text{on }\partial \Omega,\\
\int_\Omega |\varphi_0(x)|^2\,dx=1.
 \end{cases}
\end{equation}
From \cite[Theorem 1.3]{FFT} (see also  \cite[Theorem 1.5]{NT} and
\cite[Proposition 2.1]{AF}) it is known that
\begin{equation}\label{eq:37}
  \varphi_0 \text{ has at $0$ a zero
    or order $\frac k2$ for some odd $k\in \N$},
\end{equation}
and there exist $\beta_1,\beta_2\in\C$ such that
$(\beta_1,\beta_2)\neq(0,0)$ and
\begin{equation}\label{eq:131}
  r^{-k/2} \varphi_0(r(\cos t,\sin t)) \to 
  \beta_1
  \frac{e^{i\frac t2}}{\sqrt{\pi}}\cos\Big(\frac k2
  t\Big)+\beta_2 
  \frac{e^{i\frac t2}}{\sqrt{\pi}}\sin\Big(\frac k2
  t\Big) \quad \text{in }C^{1,\tau}([0,2\pi],\C)
\end{equation}
as $r\to0^+$ for any $\tau\in (0,1)$.
We recall that, by \cite{HHOO99} (see also \cite[Lemma 2.3]{BNNNT}),
the function $e^{-i\frac t2}\varphi_0(r(\cos t,\sin t))$
is a multiple of a real-valued function and therefore 
either
$\beta_1=0$ or
$\tfrac{\beta_2}{\beta_1}$
is real.
Then $\varphi_0$ has exactly $k$ nodal lines meeting
at $0$ and dividing the whole angle into $k$ equal parts; such nodal
lines are tangent to the $k$ half-lines $\big\{\big(t,\tan
(\alpha_0+j\frac{2\pi}k) t \big):\,t>0 \big\}$, $j=0,1,\dots,k-1$, where
\begin{equation}\label{eq:alpha0}
\alpha_0=
\begin{cases}
  \frac2k\mathop{\rm arccot}\big(-\frac{\beta_2}{\beta_1}\big),&\text{if
  }\beta_1\neq0,\\
0,&\text{if
  }\beta_1=0.
\end{cases}
\end{equation}
At a deeper study, the rate of convergence of $\lambda_a$ to $\lambda_0$
is strictly related to the number of nodal lines of $\varphi_0$ ending at $0$.
First results in this direction are proved in \cite{BNNNT}, in which the authors provide 
some estimates for the rate of convergence \eqref{eq:conv_auto}. 
A significant improvement of these studies is obtained in \cite{AF}, where 
sharp
asymptotic behavior of eigenvalues is provided as the pole is approaching an internal zero of an eigenfunction $\varphi_0$ of the limiting
problem \eqref{eq:equation_lambda0} along the half-line tangent to any
nodal line of $\varphi_0$; more precisely, in \cite[Theorem 1.2]{AF}
it is proved that, under assumptions \eqref{eq:1} and \eqref{eq:37}, the limit 
\begin{equation}\label{eq:limlavve}
\lim_{|a|\to 0^+}\frac{\lambda_0-\lambda_a}{|a|^{k}} \text{ is finite
  and strictly positive as $a\to0$ tangentially to a nodal line}.
\end{equation}
More precisely, the above positive limit can be expressed in terms of the value
${\mathfrak m}_k$ defined as follows. 
Let $s_0$ be the positive half-axis
$s_0\!=\![0,+\infty)\!\times\!\{0\}$.  For every odd
natural number $k$, 
the function
\begin{equation}\label{eq:psi_k}
  \psi_k(r\cos t,r\sin t)= r^{k/2} \sin
  \bigg(\frac{k}{2}\,t\bigg),\quad 
  r\geq0,\quad t\in[0,2\pi],
\end{equation}
is the unique (up to a multiplicative
constant) function which is harmonic on $\R^2\setminus s_0$,
homogeneous of degree $k/2$ and vanishing on $s_0$. 
Let $s:=\{(x_1,x_2)\in\R^2: x_2=0\text{ and }x_1\geq 1\}$,
$\R^2_+=\{(x_1,x_2)\in\R^2:x_2>0)\}$, and denote as $\Di_s(\R^2_+)$
the completion of $C^\infty_{\rm c}(\overline{\R^2_+} \setminus s)$
under the norm $( \int_{\R^2_+} |\nabla u|^2\,dx )^{1/2}$. By
standard minimization methods, the functional
\begin{equation*}
 J_k: \Di_{s}(\R^2_+)\to\R,\quad 
J_k(u) = \frac12 \int_{\R^2_+} |\nabla u(x)|^2 \,dx-
 \int_{ \partial
 \R^2_+\setminus s} u(x_1,0)\frac{\partial \psi_k}{\partial x_2}(x_1,0)\,dx_1,
\end{equation*}
achieves its minimum over the
whole space $\Di_{s}(\R^2_+)$ at some function $w_k\in
\Di_{s}(\R^2_+)$, i.e.  there exists $w_k\in \Di_{s}(\R^2_+)$ such
that
\begin{equation}\label{eq:Ik}
{\mathfrak m}_k=\min_{u\in \Di_{s}(\R^2_+)}J_k(u)=J_k(w_k).
\end{equation}
We notice that 
\begin{equation}\label{eq:segno_mk}
{\mathfrak m}_k=J_k(w_k)=-\frac12  \int_{\R^2_+} |\nabla
w_k(x)|^2 \,dx=-\frac12 \int_0^1 \dfrac{\partial_+ \psi_k}{\partial x_2}(x_1,0)\,w_k (x_1,0)\,dx_1  
<0,
\end{equation}
where, for all $x_1>0$, $\frac{\partial_+ \psi_k}{\partial
  x_2}(x_1,0)=\lim_{t\to0^+}\frac{\psi_k(x_1,t)-\psi_k(x_1,0)}{t}=\frac
k2 x_1^{\frac k2-1}$.
In \cite{AF} it is proved that the limit in \eqref{eq:limlavve} is
equal to $-\frac4\pi(|\beta_1|^2+|\beta_2|^2)\,{\mathfrak m}_k$ 
with $(\beta_1,\beta_2)\neq(0,0)$ being as in \eqref{eq:131}.

From \cite[Theorem 1.2]{AF} we can easily deduce that, under assumptions \eqref{eq:1} and \eqref{eq:37}, the Taylor polynomials of the function
$a\mapsto \lambda_0-\lambda_a$ with center $0$ and degree strictly
smaller than $k$ vanish.
\begin{Lemma}\label{l:taylor}
  Let
$\Omega\subset\R^2$ be a bounded, open and simply connected
domain such that $0\in\Omega$ and let $n_0\geq 1$ be such that 
the $n_0$-th eigenvalue $\lambda_0=\lambda_{n_0}^0$ of $(i\nabla +
A_0)^2$ on $\Omega$ is simple  with associated eigenfunctions having  in
  $0$ a zero of order $k/2$ with $k\in\N$ odd. For $a\in\Omega$ let $\lambda_a=\lambda_{n_0}^a$ be
  the $n_0$-th eigenvalue of $(i\nabla +
A_a)^2$ on $\Omega$. 
Then
\begin{equation}\label{eq:tay}
 \lambda_0-\lambda_a=P(a)+o(|a|^k),\quad\text{as }|a|\to0^+,
\end{equation}
for some homogeneous polynomial $P\not\equiv 0$ of degree $k$
\begin{equation}\label{eq:polinomio}
P(a)=P(a_1,a_2)=\sum_{j=0}^k c_j a_1^{k-j}
a_2^j.
\end{equation}
\end{Lemma}

 The aim of the present paper is to detect the exact value
of all coefficients of the polynomial $P$ (and hence the sharp
asymptotic behavior of $\lambda_a-\lambda_0$  
as $a\to0$ along any direction, see Figure \ref{fig:1}).

\begin{figure}\centering
\begin{tikzpicture}[scale=1.2]
    \draw[line width=1.5pt] (0,0) to [out=0, in=-120] (1.732,1);
    \draw[line width=1.5pt] (0,0) to [out=120, in=0] (-1.732,1);
    \draw[line width=1.5pt] (0,0) to [out=240, in=120] (0,-2);
\fill (0,0) circle (2pt) node[above] {$0$};
\draw[dotted,line width=1pt]  (0,0) -- (2,1);
\draw[dotted,line width=1pt]  (0,0) -- (2,0);
\draw[line width=2pt,->] (1.2,0.6) -- (0.6,0.3);
\draw[dotted, line width=1pt,-] (1.341,0) to [out=90, in=-50] (1.2,0.6);
\node at (1.55,0.2) {${\alpha}$};
\fill (1.2,0.6) circle (2pt) node[above] {$a$};
\node at (2.1,1) {${\mathfrak r}$};
                              \end{tikzpicture}
\caption{$a=|a|(\cos\alpha,\sin\alpha)$ approaches $0$ along the
  direction determined by the angle $\alpha$.}
\label{fig:1}
\end{figure}
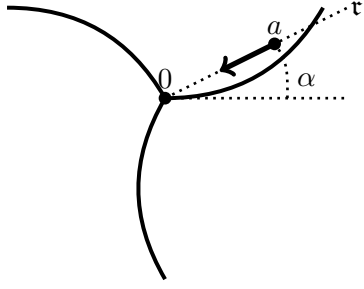

\begin{Theorem}\label{t:main}
Let
$\Omega\subset\R^2$ be a bounded, open and simply connected
domain such that $0\in\Omega$ and let $n_0\geq 1$ be such that 
the $n_0$-th eigenvalue $\lambda_0=\lambda_{n_0}^0$ of $(i\nabla +
A_0)^2$ on $\Omega$ is simple  with associated eigenfunctions having  in
  $0$ a zero of order $k/2$ with $k\in\N$ odd. For $a\in\Omega$ let $\lambda_a=\lambda_{n_0}^a$ be
  the $n_0$-th eigenvalue of $(i\nabla +
A_a)^2$ on $\Omega$. 
Let $\alpha\in[0,2\pi)$.
Then  
\[
\frac{\lambda_0-\lambda_a}{|a|^k}\to C_0
\cos\big(k(\alpha-\alpha_0)\big) 
 \qquad \text{as $a\to0$ with $a=|a|(\cos\alpha,\sin\alpha)$,}
\]
where $\alpha_0$ is 
 as in \eqref{eq:alpha0} and 
\[
C_0=-4\frac{|\beta_1|^2+|\beta_2|^2}{\pi}{\mathfrak m}_k,
\]
with
$(\beta_1,\beta_2)\neq(0,0)$ as in \eqref{eq:131} and ${\mathfrak
  m}_k$ as in \eqref{eq:Ik}-\eqref{eq:segno_mk}.
\end{Theorem}

\begin{remark}
  By Theorem \ref{t:main} it follows that the polynomial 
  \eqref{eq:polinomio} of Lemma \ref{l:taylor} is  given by
 \[
  P(|a|(\cos\alpha,\sin\alpha))=C_0|a|^k\cos(k(\alpha-\alpha_0)).
 \]
Hence 
\[
P(a_1,a_2)=C_0\Re\big(e^{-ik\alpha_0}(a_1+i\,a_2)^k\big),
\]
thus yielding $\Delta P=0$, i.e. the polynomial $P$ in
\eqref{eq:tay}-\eqref{eq:polinomio} is \emph{harmonic}.
\end{remark}

The proof of Theorem \ref{t:main} is based on a combination
of estimates from
above and below of the Rayleigh quotient associated to the eigenvalue
problem with a fine blow-up analysis for scaled
eigenfunctions 
\[
\dfrac{\varphi_a(|a|x)}{|a|^{k/2}}
\]
which gives a sharp characterization of  upper
and lower bounds for eigenvalues, as already performed in \cite{AF}. 
Nevertheless, differently from \cite{AF}, in the general case of poles moving along
any direction, we cannot explicitly construct the
limit profile of the above blow-up sequence. Such a difficulty
is overcome studying the dependence of the limit profile on the
position of the pole and the symmetry/periodicity  properties of its Fourier
coefficient with respect to a basis of eigenvectors of an associated
angular problem: such symmetry and periodicity turn into some  
 symmetry and periodicity invariances of the polynomial $P$. A
 complete classification of homogeneous $k$-degree polynomials with
 such periodicity/symmetry invariances allows us to determine
 explicitly the polynomial $P$ thus concluding.

The paper is organized as follows.  Section \ref{sec:preliminaries} is
devoted to recall some known facts and introduce some notation.  In
section \ref{sec:sharp-asympt-lambd} we prove sharp asymptotics for
$\lambda_0-\lambda_a$ in dependence of the angle $\alpha$.  In section
\ref{sec:properties-falpha} we describe some symmetry properties of
the sharp asymptotics, which allow us to prove Theorem \ref{t:main} in
section \ref{sec:proof-main-result}.

\section{Preliminaries}\label{sec:preliminaries}

In this  section we present some preliminaries as
needed in the forthcoming argument.

\subsection{Change of coordinates}
 As already highlighted in \cite[Remark 2.2]{AF}, up to a change
 of coordinates (a rotation), it is not restrictive to assume in 
 \eqref{eq:131} that 
 \begin{equation}\label{eq:54}
  \beta_1=0.
\end{equation}
Under condition \eqref{eq:54}, we have that $\alpha_0=0$ and one 
nodal line of $\varphi_0$ is tangent the $x_1$-axis.

\subsection{Polar eigenfunctions} 
The limit function in \eqref{eq:131} is an eigenfunction of the operator
\begin{equation*}
\mathfrak L\psi=
-\psi''+i\psi'+\frac14\psi
\end{equation*}
 acting on
$2\pi$-periodic functions.
The eigenvalues of $\mathfrak L$ are
$\big\{\frac {j^2}4:j\in \N,\ j\text{ is odd}\big\}$; moreover each
eigenvalue $\frac{j^2}4$ has multiplicity $2$ and the functions
\begin{equation}\label{eq:9}
  \psi_1^j(t)=\frac{e^{i\frac t2}}{\sqrt{\pi}}\cos\Big(\frac j2
  t\Big),\quad 
  \psi_2^j(t)=\frac{e^{i\frac t2}}{\sqrt{\pi}}\sin\Big(\frac j2
  t\Big)
\end{equation} 
form an
$L^2((0,2\pi),\C)$-orthonormal basis of the eigenspace associated to
the eigenvalue $\frac{j^2}4$.

\subsection{Angles and approximating eigenfunctions}
As in \cite{BNNNT}, for every $\alpha\in[0,2\pi)$  and 
$b=(b_1,b_2)=|b|(\cos\alpha,\sin\alpha)\in \R^2\setminus\{0\}$, we define
\begin{equation}\label{eq:4th}
\theta_b:\R^2\setminus\{b\}\to [\alpha,\alpha+2\pi)
\quad\text{and}\quad
\theta_0^b:\R^2\setminus\{0\}\to [\alpha,\alpha+2\pi)
\end{equation}
such that 
\[
\theta_b(b+r(\cos t,\sin t))=t
\quad \text{and}\quad \theta_0^b(r(\cos t,\sin t))=t,
\quad\text{for all }r>0\text{ and }t\in  [\alpha,\alpha+2\pi). 
\]
E.g. 
if $b_1>0$ and $b_2>0$ the functions $\theta_b$ and $\theta_0^b$ are given by
\begin{align}\label{eq:theta}
&\theta_b(x_1,x_2)=
\begin{cases}
  \arctan\frac{x_2-b_2}{x_1-b_1},&\text{if }x_1>b_1,\ x_2\geq \frac{b_2}{b_1}x_1,\\ 
  \frac\pi2,&\text{if }x_1=b_1,\ x_2> b_2,\\ 
  \pi+\arctan\frac{x_2-b_2}{x_1-b_1},&\text{if }x_1<b_1,\\ 
  \frac32\pi,&\text{if }x_1=b_1,\ x_2< b_2,\\ 
  2\pi+\arctan\frac{x_2-b_2}{x_1-b_1},&\text{if }x_1>b_1,\ x_2< \frac{b_2}{b_1}x_1,
\end{cases}\\
\notag
&\theta_0^b(x_1,x_2)=
\begin{cases}
  \arctan\frac{x_2}{x_1},&\text{if }x_1>0,\ x_2\geq \frac{b_2}{b_1}x_1,\\ 
  \frac\pi2,&\text{if }x_1=0,\ x_2> 0,\\ 
  \pi+\arctan\frac{x_2}{x_1},&\text{if }x_1<0,\\ 
  \frac32\pi,&\text{if }x_1=0,\ x_2< 0,\\ 
  2\pi+\arctan\frac{x_2}{x_1},&\text{if }x_1>0,\ x_2< \frac{b_2}{b_1}x_1.
\end{cases}
\end{align}
We notice that $\theta_b$ and $\theta_0^b$ are regular except on the half-lines 
\begin{equation*}
 s_b:= \big\{tb:\,t\geq 1\big\}, \quad 
 s_0^b:=  \big\{t\,b:\,t\geq 0\big\},
\end{equation*}
respectively,
whereas the difference $\theta_0^b - \theta_b$ is regular except for the segment 
$\{ tb:\ t\in [0,1] \}$ from $0$ to $b$.

We also define
\[
\theta_0:\R^2\setminus\{0\}\to [0,2\pi)\]
 as 
\begin{equation}\label{eq:tildetheta}
\theta_0(x_1,x_2)=
\begin{cases}
  \arctan\frac{x_2}{x_1},&\text{if }x_1>0,\ x_2\geq 0,\\ 
  \frac\pi2,&\text{if }x_1=0,\ x_2> 0,\\ 
  \pi+\arctan\frac{x_2}{x_1},&\text{if }x_1<0,\\ 
  \frac32\pi,&\text{if }x_1=0,\ x_2< 0,\\ 
  2\pi+\arctan\frac{x_2}{x_1},&\text{if }x_1>0,\ x_2<0,\\ 
\end{cases}
\end{equation}
so that $\theta_0(\cos t,\sin t)=\theta_0^0(\cos t,\sin t) =t$ for all
$t\in[0,2\pi)$ and $\theta_0$ is regular except for the half-axis
$\{(x_1,0): x_1\geq 0\}$,
whereas the difference 
\begin{equation}\label{eq:diff}
(\theta_0^b - \theta_0) (r\cos t,r\sin t) = 
 \begin{cases}
  0, &\text{if }t\in [\alpha, 2\pi),\\
  2\pi, &\text{if } t\in[0,\alpha).
 \end{cases}
\end{equation}
Let us now consider a suitable family of eigenfunctions relative to
the approximating eigenvalue $\lambda_a$.
For all $a\in\Omega$, let $\varphi_a\in
H^{1,a}_{0}(\Omega,\C)\setminus\{0\}$ be an eigenfunction of problem
\eqref{eq:eige_equation_a} associated to the eigenvalue $\lambda_a$,
i.e. solving
\begin{equation}\label{eq:equation_a}
 \begin{cases}
   (i\nabla + A_a)^2 \varphi_a = \lambda_a \varphi_a,  &\text{in }\Omega,\\
   \varphi_a = 0, &\text{on }\partial \Omega,
 \end{cases}
\end{equation}
such that 
\begin{equation}\label{eq:6}
  \int_\Omega |\varphi_a(x)|^2\,dx=1 \quad\text{and}\quad 
  \int_\Omega e^{\frac i2(\theta_0^a-\theta_a)(x)}\varphi_a(x)\overline{\varphi_0(x)}\,dx\text{ is a
    positive real number},
\end{equation}
where $\varphi_0$ is as in \eqref{eq:equation_lambda0}.
From \eqref{eq:1}, \eqref{eq:equation_lambda0},
\eqref{eq:equation_a}, \eqref{eq:6}, and standard elliptic estimates,
it follows that $\varphi_a\to \varphi_0$ in $H^1(\Omega,\C)$ and in
$C^2_{\rm loc}(\Omega\setminus\{0\},\C)$
and 
\begin{equation}\label{eq:congrad2}
(i\nabla+A_a)\varphi_a\to (i\nabla+A_0)\varphi_0 \quad\text{in }L^2(\Omega,\C).
\end{equation}

\subsection{Limit profile in dependence on $\alpha$} 
A key role in the proof of our main result is played by
a suitable magnetic-harmonic function in $\R^2$, which will turn out
to be the limit of
blowed-up sequences of eigenfunctions with poles approaching $0$ along the half-line
starting from $0$ with slope $\tan \alpha$.

For every ${\mathbf p}\in\R^2$, we denote as ${\mathcal
  D}^{1,2}_{\mathbf p}(\R^2,\C)$ the completion of $C^\infty_{\rm
  c}(\R^N\setminus\{0\},\C)$ with respect to the magnetic Dirichlet
norm
\begin{displaymath}
  \|u\|_{{\mathcal D}^{1,2}_{\mathbf
      p}(\R^2,\C)}:=\bigg(\int_{\R^2}\big|(i\nabla +A_{\mathbf p})u(x)\big|^2\,dx\bigg)^{\!\!1/2}.
\end{displaymath}
We recall from \cite{LW99} that functions in ${\mathcal D}^{1,2}_{\mathbf p}(\R^2,\C)$
satisfy the Hardy type inequality 
\begin{equation*}
  \int_{\R^2} |(i\nabla+A_{\mathbf p})u|^2\,dx \geq \frac14
  \int_{\R^2}\frac{|u(x)|^2}{|x-{\mathbf p}|^2}\,dx;
\end{equation*}
furthermore  (see also \cite[Lemma 3.1 and Remark 3.2]{FFT}) the inequality
\begin{equation*} 
 \int_{D_r({\mathbf p})} |(i\nabla+A_{\mathbf p})u|^2\,dx \geq \frac14 \int_{D_r({\mathbf p})}\frac{|u(x)|^2}{|x-{\mathbf p}|^2}\,dx,
\end{equation*}
 holds for all $r>0$  and $u\in H^{1,{\mathbf p}}(D_r({\mathbf p}),\C)$, where 
$D_r({\mathbf p})$
  denotes the disk of center ${\mathbf p}$ and radius $r$.

\begin{Proposition}\label{prop_Psi}
Let $\alpha\in[0,2\pi)$ and ${\mathbf
  p}=(\cos\alpha,\sin\alpha)$. There exists a unique function
$\Psi_{\mathbf p}\in H^{1 ,{\mathbf p}}_{\rm loc}(\R^2,\C)$ such that 
\begin{equation}\label{eq:16}
(i\nabla +A_{\mathbf p})^2\Psi_{\mathbf p}=0\quad\text{ in $\R^2$ in a
  weak $H^{1 ,{\mathbf p}}$-sense},
\end{equation}
and 
\begin{equation}\label{eq:17}
\int_{\R^2\setminus D_r} \big|(i\nabla + A_{\mathbf p})(\Psi_{\mathbf p} -
e^{\frac i2(\theta_{\mathbf p}-\theta_0^{\mathbf p})} e^{\frac i2
  \theta_0}\psi_k )\big|^2\,dx < +\infty,
\quad\text{for any }r>1,
\end{equation}
where $D_r=D_r(0)$.
\end{Proposition}
\begin{proof}
Let $\eta $ be a smooth cut-off function such that $\eta \equiv 0$ in
$D_1$ and $\eta\equiv 1$ in $\R^2 \setminus D_R$ for some $R>1$. 
We observe that 
\begin{align*}
  F&=(\Delta \eta) e^{\frac i2(\theta_{\mathbf p}-\theta_0^{\mathbf p})}
  e^{\frac i2 \theta_0}\psi_k-2i\nabla\eta\cdot(i\nabla+A_{\mathbf
    p})\Big(e^{\frac i2(\theta_{\mathbf p}-\theta_0^{\mathbf p})} e^{\frac
    i2
    \theta_0}\psi_k\Big)\\
  &=-(i\nabla+A_{\mathbf p})^2\Big(\eta e^{\frac i2(\theta_{\mathbf
      p}-\theta_0^{\mathbf p})} e^{\frac i2 \theta_0}\psi_k\Big)\in
  \big({\mathcal D}^{1,2}_{\mathbf p}(\R^2,\C)\big)^\star. 
\end{align*}
Hence, via Lax-Milgram's Theorem  there exists a unique solution $g\in
{\mathcal D}^{1,2}_{\mathbf p}(\R^2,\C)$  to problem 
\begin{equation*}
  (i\nabla + A_{\mathbf p})^2 g = F, \quad\text{in }\big({\mathcal D}^{1,2}_{\mathbf p}(\R^2,\C)\big)^\star.
\end{equation*}
The function $\Psi_{\mathbf p}=g+\eta e^{\frac i2(\theta_{\mathbf
      p}-\theta_0^{\mathbf p})} e^{\frac i2 \theta_0}\psi_k$ satisfies
  \eqref{eq:16} and \eqref{eq:17}. 

  To prove uniqueness, it is enough to observe that, if two function
  $\Psi_{\mathbf p}^1,\Psi_{\mathbf p}^2\in H^{1 ,{\mathbf p}}_{\rm
    loc}(\R^2,\C)$ satisfy \eqref{eq:16} and \eqref{eq:17}, then their
  difference $\Psi_{\mathbf p}^1-\Psi_{\mathbf p}^2$ belongs to the space ${\mathcal
    D}^{1,2}_{\mathbf p}(\R^2,\C)$ (see \cite[Proposition 4.3]{AF});
  since $(i\nabla + A_{\mathbf p})^2(\Psi_{\mathbf p}^1-\Psi_{\mathbf p}^2)=0$
  in $({\mathcal D}^{1,2}_{\mathbf p}(\R^2,\C))^\star$, we
  conclude that necessarily $\Psi_{\mathbf p}^1-\Psi_{\mathbf p}^2\equiv 0$.
\end{proof}
\begin{remark}\label{rem:dec_infty}
We observe that from \cite[Theorem 1.5]{FFT} it follows easily that 
\[
\Psi_{\mathbf p} -
e^{\frac i2(\theta_{\mathbf p}-\theta_0^{\mathbf p})} e^{\frac i2
  \theta_0}\psi_k=O(|x|^{-1/2}),\quad\text{as }|x|\to+\infty.
\]
\end{remark}

\section{Sharp asymptotics for $\lambda_0-\lambda_a$ in dependence on $\alpha$}\label{sec:sharp-asympt-lambd}

As in \cite{AF}, the argument relies essentially 
on the
Courant-Fisher minimax characterization of eigenvalues. The
asymptotics for eigenvalues is derived by combining estimates from
above and below of the Rayleigh quotient, obtained by using  as
test functions 
suitable manipulations of eigenfunctions. In this way, we prove upper
and lower bounds whose limit as $a\to0$ can be described in terms of
the limit profile constructed in Proposition \ref{prop_Psi}.

For all $1\leq j\leq n_0$ and $a\in\Omega$, let $\varphi_j^a\in
H^{1,a}_{0}(\Omega,\C)\setminus\{0\}$ be an eigenfunction of problem
\eqref{eq:eige_equation_a} associated to the eigenvalue $\lambda_j^a$,
i.e. solving
\begin{equation}\label{eq_eigenfunction}
 \begin{cases}
   (i\nabla + A_a)^2 \varphi_j^a = \lambda_j^a \varphi_j^a,  &\text{in }\Omega,\\
   \varphi_j^a = 0, &\text{on }\partial \Omega,
 \end{cases}
\end{equation}
such that 
\begin{equation}\label{eq:23}
  \int_\Omega |\varphi_j^a(x)|^2\,dx=1\quad\text{and}\quad 
  \int_\Omega \varphi_j^a(x)\overline{\varphi_\ell^a(x)}\,dx=0\text{ if }j\neq\ell.
\end{equation}
For $j=n_0$, we choose 
\begin{equation}\label{eq:89}
  \varphi_{n_0}^a=\varphi_a,
\end{equation}
with $\varphi_a$ as in \eqref{eq:equation_a}--\eqref{eq:6}.

\subsection{Rayleigh quotient for $\lambda_a$}
We revisit \cite[Subsection 6.2]{AF}, whose  main strategy does not differ in this case; 
however,  it is worth presenting here the key points, in order to highlight the dependence of the result 
on the position of the blown-up pole. 

By the Courant-Fisher \emph{minimax characterization} of the eigenvalue
 $\lambda_a$, we have that
\begin{equation}\label{eq:91_la}
  \lambda_a =\! \min\bigg\{\!\max_{u\in F\setminus \{0\}}
\dfrac{\int_{\Omega} \abs{(i\nabla+A_a) u}^2dx}{\int_{\Omega}
  |u|^2\,dx}:F \text{ is a subspace of $H^{1,a}_0(\Omega,\C)$, $\dim F= n_0$}\bigg\}.
\end{equation}
Let  $R>2$ and $\alpha\in[0,2\pi)$. For $a=|a|(\cos\alpha,\sin\alpha)$
with $|a|$ sufficiently small,
we consider the functions
$w_{j,R,a}$ defined  as
\[
 w_{j,R,a}= 
 \begin{cases}
  w_{j,R,a}^{ext}, &\text{in }\Omega \setminus D_{R|a|},\\
  w_{j,R,a}^{int}, &\text{in } D_{R|a|},
 \end{cases}
\quad j=1,\ldots,n_0,
\]
where 
\begin{equation*}
  w_{j,R,a}^{ext} := e^{\frac{i}{2}(\theta_a - \theta_0^a)}
  \varphi_j^0\quad \text{in }
  \Omega \setminus D_{R|a|}, 
\end{equation*}
with $\varphi_j^0$ as in \eqref{eq_eigenfunction}--\eqref{eq:89} with $a=0$,  so that it solves
\begin{equation*}
 \begin{cases}
   (i\nabla +A_a)^2 w_{j,R,a}^{ext} = \lambda_j^0 w_{j,R,a}^{ext}, &\text{in }\Omega \setminus D_{R|a|},\\
   w_{j,R,a}^{ext} = e^{\frac{i}{2}(\theta_a - \theta_0^a)} \varphi_j^0,
   &\text{on }\partial (\Omega \setminus D_{R|a|}),
 \end{cases}
\end{equation*}
whereas $w_{j,R,a}^{int}$ is the unique solution to the minimization problem 
\begin{multline*}
  \int_{D_{R|a|}} |(i\nabla +A_a)
  w_{j,R,a}^{int}(x)|^2\,dx\\
  = \min\left\{ \int_{D_{R|a|}} |(i\nabla +A_a)u(x)|^2\,dx:\, u\in
    H^{1,a}(D_{R|a|},\C), \ u= e^{\frac{i}{2}(\theta_a - \theta_0^a)}
    \varphi_j^0 \text{ on }\partial D_{R|a|} \right\},
\end{multline*}
thus solving
\begin{equation*}
 \begin{cases}
  (i\nabla +A_a)^2 w_{j,R,a}^{int} = 0, &\text{in }D_{R|a|},\\
  w_{j,R,a}^{int} = e^{\frac{i}{2}(\theta_a - \theta_0^a)} \varphi_j^0, &\text{on }\partial D_{R|a|}.
 \end{cases}
\end{equation*}
Following the argument  in \cite[Subsection 6.2]{AF},
letting ${\mathbf p}=(\cos \alpha,\sin\alpha)$, we now define the function 
$w_R$ as the unique solution to the minimization problem 
\begin{multline*}
  \int_{D_{R}} |(i\nabla +A_{\mathbf p})
  w_R(x)|^2\,dx\\
  = \min\left\{ \int_{D_{R}} |(i\nabla +A_{\mathbf p})u(x)|^2\,dx:\, u\in
    H^{1,{\mathbf p}}(D_{R},\C), \ u= e^{\frac{i}{2}(\theta_{\mathbf p}-\theta_{0}^{\mathbf p})}e^{\frac{i}{2}\theta_{0}}
\psi_k \text{ on }\partial D_{R} \right\},
\end{multline*}
which then  solves
\begin{equation}\label{eq:equazione_w_R}
 \begin{cases}
  (i\nabla +A_{\mathbf p})^2 w_R = 0, &\text{in }D_{R},\\
  w_R = e^{\frac{i}{2}(\theta_{\mathbf p}-\theta_{0}^{\mathbf p})}e^{\frac{i}{2}\theta_{0}}
\psi_k, &\text{on }\partial D_{R}.
 \end{cases}
\end{equation}
We also introduce the following blow-up sequences as in \cite{AF}:
\begin{align}
U_a^R(x):=\frac{w_{n_0,R,a}^{int} (|a|x)}{|a|^{k/2}}, \quad 
 W_a(x):=\frac{\varphi_0(|a|x)}{|a|^{k/2}}.
\end{align}
As in \cite{AF}, under assumptions \eqref{eq:37} and \eqref{eq:54}, from \cite[Theorem 1.3 and
Lemma 6.1]{FFT} we have  that 
\begin{equation}\label{eq:vkext_la}
W_a\to  \beta e^{\frac i2\theta_0}\psi_k\quad\text{as } |a|\to0
\end{equation}
in $H^{1 ,0}(D_R,\C)$ for
every $R>1$, where $\psi_k$ is defined in \eqref{eq:psi_k} and 
\begin{equation}\label{eq:beta}
\beta:= \frac{\beta_2}{\sqrt\pi}
\end{equation}
with $\beta_2$ as in \eqref{eq:131}.
On the other hand, we now meet the following differences with respect to the case $\alpha=0$ studied in \cite{AF}:
\begin{itemize}
 \item By the Dirichlet principle and \eqref{eq:vkext_la}, we have that  
\begin{align*}
 \int_{D_R}& \abs{(i\nabla +A_{\mathbf p})(U_a^R - \beta w_R) }^2dx\\
&\leq \int_{D_R} \abs{(i\nabla +A_{\mathbf p})\big(\eta_R\,
e^{\frac{i}{2}(\theta_{\mathbf p}-\theta_0^{\mathbf p})}(W_a- \beta e^{\frac{i}{2}\theta_{0} }\psi_k\big)}^2 dx\\
&\leq 2 \int_{D_R}|\nabla \eta_R|^2\big|W_a- \beta e^{\frac{i}{2}\theta_{0} }\psi_k\big|^2dx \\
&\quad+2\int_{D_R\setminus D_{R/2}}
\eta_R^2\big|(i\nabla +A_{0})(W_a- \beta e^{\frac{i}{2}\theta_{0} }\psi_k)\big|^2dx
= o(1)\quad \text{as }
|a|\to0^+,
\end{align*}
where $\eta_R:\R^2\to\R$ is a smooth cut-off function such that
\begin{equation*}
  \eta_R\equiv 0\text{ in }D_{R/2},\quad 
  \eta_R\equiv 1 \text{ on }\R^2\setminus D_{R},\quad
  |\nabla\eta_R|\leq4/R\text{ in }\R^2.
\end{equation*}
Hence, for all $R>2$,
$U_a^R \to  \beta w_R$ in  $H^{1,{\mathbf p}}(D_R,\C)$
as $a=|a|{\mathbf p}\to 0$, where $\beta$ is defined in \eqref{eq:beta}.
 \item For every $r>1$, $w_R\to\Psi_{\mathbf p}$ in $H^{1,{\mathbf p}}(D_r,\C)$ as $R\to+\infty$.  
 This follows as in the proof of \cite[Lemma 6.5]{AF} up to suitable
 obvious modifications and taking into account \eqref{eq:17} and
 Remark \ref{rem:dec_infty}.
\end{itemize}
Taking into account what observed above and using (after a
Gram-Schmidt normalization) the functions $w_{j,R,a}$ as test functions in
the Rayleigh quotient, we can argue as in 
\cite[Lemma 6.6]{AF} to obtain the following estimate.
\begin{Lemma}\label{l:stima_Lambda0_sotto}
 For $\alpha\in[0,2\pi)$ and $a=|a|(\cos\alpha,\sin\alpha)\in\Omega$, let 
$\lambda_a\in\R$ and $\varphi_a\in
H^{1,a}_{0}(\Omega,\C)$ solve (\ref{eq:equation_a}-\ref{eq:6})
and $\lambda_0\in\R$ and $\varphi_0\in
H^{1,0}_{0}(\Omega,\C)$ solve
\eqref{eq:equation_lambda0}. If \eqref{eq:1} and
\eqref{eq:37} hold and \eqref{eq:54} is satisfied, then, for all
  $R>\tilde R$ and $a=|a|(\cos\alpha,\sin\alpha)\in\Omega$,
\[
\frac{\lambda_0-\lambda_a}{|a|^k}\geq g_R(a)
\]
where 
\[
\lim_{|a|\to 0}g_R(a)=i|\beta|^2\tilde\kappa_R,
\]
with $\beta$ as in \eqref{eq:beta} and 
\begin{equation}\label{eq:tildekappa_R}
\tilde\kappa_R=\int_{\partial D_R} \Big(e^{-\frac i2 \theta_{\mathbf
    p}}e^{\frac i2 (\theta_{0}^{\mathbf p}-\theta_0)} (i\nabla+A_{\mathbf p})w_R\cdot\nu 
- (i\nabla) \psi_k\cdot\nu \Big) \psi_k\,ds
\end{equation}
being ${\mathbf p}=(\cos\alpha,\sin\alpha)$ and $\psi_k$ as in \eqref{eq:psi_k}.
\end{Lemma}
\begin{proof}
 The proof follows exactly as in \cite[Lemma 6.6]{AF}, so we omit it. 
\end{proof}

For any $R>1$ let us introduce the following Fourier-type coefficient
\begin{equation}\label{eq:def_upsilon}
 \upsilon_R(r):= \int_{0}^{2\pi} e^{-\frac i2 \theta_{\mathbf p}(r\cos t,r\sin t)} w_R(r\cos t,r\sin t)
 e^{\frac i2 \theta_0^{\mathbf p} (r\cos t,r\sin t)}\overline{\psi_2^k(t)}, \quad r\in[1,R],
\end{equation}
with $\psi_2^k$ defined in \eqref{eq:9}.

\begin{Lemma}\label{l:upsilon}
 For any $R>1$ the function $\upsilon_R$ defined in \eqref{eq:def_upsilon} satisfies 
 \begin{equation}\label{eq:upsilon}
  \big( r^{-k/2}\upsilon_R(r) \big)'= \frac{c_{R}}{r^{1+k}}, \quad \text{in }(1,R),
 \end{equation} 
 for some $c_{R}\in\C$.
\end{Lemma}
\begin{proof}
To prove \eqref{eq:upsilon} it is enough to show that 
\[
\int_1^R\bigg(-\upsilon_R'' - \frac1r \upsilon_R' +
\frac{k^2}{4r^2}\upsilon_R \bigg)r\eta(r)\,dr =0, \quad \text{for all
}\eta\in C^\infty_{\rm c}(1,R).
\]
 By \eqref{eq:equazione_w_R}, it is easy to see that 
 the function $u(x):= e^{-\frac i2 \theta_{\mathbf p}(x)}w_R(x)$ is
 harmonic in $D_R \setminus s_{\mathbf p}$.
  Let us consider an arbitrary function $\eta(r)\in C^\infty_{\rm c}(1,R)$ and the function
 \[
  g(t):=\frac1{\sqrt\pi}e^{\frac i2 (\theta_0^{\mathbf p}-\theta_0)(\cos t,\sin t)} \sin(\tfrac k2 t)=
  \begin{cases}
   -\frac1{\sqrt\pi}\sin(\tfrac k2 t) & t\in[0,\alpha)\\
   \frac1{\sqrt\pi}\sin(\tfrac k2 t) & t\in[\alpha,2\pi).
  \end{cases}
 \]
Testing equation $-\Delta u=0$ with $v(r\cos t,r\sin t)=\eta(r)g(t)$ in $D_R\setminus
s_{\mathbf p}$, integrating by parts and observing that both $v$ and $\nabla
u$ jump across $s_{\mathbf p}$, we obtain that 
\begin{align*}
  0&=\int_1^R\bigg(\int_0^{2\pi}\Big(r\partial_{r}u(r\cos t,r\sin
  t)\eta'(r)g(t)+\frac{\eta(r)}{r}g'(t)\partial_{t}u(r\cos t,r\sin
  t)\Big)\,dt\bigg)\,dr\\
&=-\int_1^R\eta(r) \bigg(\int_0^{2\pi}\Big(\partial_{r}u(r\cos t,r\sin
  t)+r\partial^2_{rr}u(r\cos t,r\sin
  t)\Big)g(t)\,dt\bigg)\,dr\\
&\qquad+
\int_1^R \frac{\eta(r)}{r}\bigg(\int_0^{2\pi}\partial_{t}u(r\cos t,r\sin
  t)g'(t)\,dt\bigg)\,dr\\
&=-\int_1^R\Big(\eta(r) \upsilon_R'(r)+r\eta(r)\upsilon_R''(r)\Big)\,dr+
\int_1^R\frac{\eta(r)}{r}\bigg(\int_0^{2\pi}\partial_{t}u(r\cos t,r\sin
  t)g'(t)\,dt\bigg)\,dr.
\end{align*}
A further integration by parts yields
\begin{align*}
  \int_0^{2\pi}&\partial_{t}u(r\cos t,r\sin
  t)g'(t)\,dt=-\int_0^{2\pi}u(r\cos t,r\sin
  t)g''(t)\,dt\\
&\qquad+g'_-(2\pi)u(r\cos(2\pi^-),r\sin(2\pi^-))-g'_+(\alpha)u(r\cos(\alpha^+),r\sin(\alpha^+))\\
&\qquad+
g'_-(\alpha)u(r\cos(\alpha^-),r\sin(\alpha^-))-
g'_+(0)u(r\cos(0^+),r\sin(0^+))\\
&\quad=-\int_0^{2\pi}u(r\cos t,r\sin
  t)g''(t)\,dt=\frac{k^2}{4}\int_0^{2\pi}u(r\cos t,r\sin
  t)g(t)\,dt=\frac{k^2}{4}\upsilon_R(r)
\end{align*}
in view of the fact that $g'_+(0)=g'_-(2\pi)$,
$g'_+(\alpha)=-g'_-(\alpha)$, and 
\[
\lim_{t\to \alpha^+}u(r\cos(t),r\sin(t))=-\lim_{t\to
  \alpha^-}u(r\cos(t),r\sin(t)).
\]
The conclusion then follows.
\end{proof}

\noindent For $\alpha\in[0,2\pi)$ and ${\mathbf
  p}=(\cos\alpha,\sin\alpha)$, let us define the following
Fourier-type coefficient of the limit profile $\Psi_{\mathbf p}$
\begin{equation}\label{eq:def_xi}
 \xi_{\mathbf p}(r):= \int_{0}^{2\pi} e^{-\frac i2 \theta_{\mathbf p}(r\cos t,r\sin t)} \Psi_{\mathbf p} (r\cos t,r\sin t)
 e^{\frac i2 \theta_0^{\mathbf p} (r\cos t,r\sin t)}\overline{\psi_2^k(t)}\,dt, \quad r\geq 1.
\end{equation}

\begin{Lemma}\label{l:lim_tilde_kappa_R}
 Let $\tilde \kappa_R$ be as in \eqref{eq:tildekappa_R}. Then
 \begin{equation*}
  \lim_{R\to+\infty} \tilde\kappa_R = ik\sqrt\pi (\sqrt\pi - \xi_{\mathbf p} (1)),
 \end{equation*}
where $\xi_{\mathbf p} (r)$ is defined in \eqref{eq:def_xi}.
\end{Lemma}
\begin{proof}
Arguing as in the proof of \cite[Lemma 6.7]{AF}, 
integrating \eqref{eq:upsilon} and taking into account the boundary
condition in \eqref{eq:equazione_w_R}, we obtain that 
\begin{equation*}
  \upsilon_R(r) =r^{k/2}
  \frac{R^k\sqrt\pi-\upsilon_R(1)}{R^k-1}-r^{-k/2}\frac{R^k(\sqrt\pi-\upsilon_R(1))}{R^k-1},
  \quad \text{for all }r\in(1,R].
\end{equation*}
By differentiation of the previous identity, we obtain that 
\begin{equation}\label{eq:127}
\upsilon_R'(R)=\frac k2\,\frac{R^{\frac{k}{2}-1}}{R^k-1}\Big((R^k+1)\sqrt\pi-2\upsilon_R(1)\Big).
\end{equation}
On the other hand, differentiation in \eqref{eq:def_upsilon} yields
\begin{equation}\label{eq:128}
  \upsilon_R'(r)=
  -\frac{i}{\sqrt\pi}r^{-1-\frac k2}
  \int_{\partial D_r}e^{-\frac{i}{2}(\theta_{\mathbf
      p}-\theta_0^{\mathbf p})}(i\nabla+A_{\mathbf p})w_R\cdot\nu\, e^{-\frac{i}{2}\theta_0}\psi_k\, ds.
\end{equation}
Combination of \eqref{eq:127}, \eqref{eq:128} and
\eqref{eq:tildekappa_R} yields that 
\[
\tilde\kappa_R=\frac{ik\sqrt\pi
  R^k}{R^k-1}\big(\sqrt\pi-\upsilon_R(1)\big).
\]
We conclude by letting $R\to+\infty$, since the convergence 
$w_R\to\Psi_{\mathbf p}$ in $H^{1,{\mathbf p}}(D_r,\C)$ as $R\to+\infty$
implies that 
$\lim_{R\to+\infty}\upsilon_R(1)=\xi_{\mathbf p} (1)$.
\end{proof}

\noindent Combining the results of Lemmas \ref{l:stima_Lambda0_sotto} and
\ref{l:lim_tilde_kappa_R} we have that 
\begin{equation}\label{eq:3.1}
\lambda_0-\lambda_a\geq |a|^k k|\beta|^2\sqrt\pi \Big(\xi_{\mathbf p} (1)-\sqrt\pi+o(1)\Big)
\end{equation}
as $a=|a|{\mathbf p}\to 0$.

\subsection{Blow-up analysis and Rayleigh quotient for $\lambda_0$}
In order to prove even an upper bound for the difference $\lambda_0-\lambda_a$ 
we refer this time to \cite[Subsection 6.1]{AF}. Differently from what occurs in
\cite{AF}, when the direction along which $a\to 0$ is not a nodal line
of $\varphi_0$ the value $(\xi_{\mathbf p} (1)-\sqrt\pi)$ can have any sign
(and vanish along some directions); this does not allow deriving
the exact asymptotic
behavior of the normalization term in the blow-up analysis from
estimates of the Rayleigh quotient from above and below as done in \cite{AF}. On the other
hand, from \cite{AF} we can derive Lemma \ref{l:taylor} and hence a
control on the size of the eigenvalue variation along any direction.

The proof of Lemma \ref{l:taylor} is based on the following result
(see also \cite[Lemma 6.6]{BNNNT}).

\begin{Lemma}\label{l:polinf}
  Let $Q(x_1,x_2)=\sum_{j=0}^h c_j x_1^j x_2^{h-j}$ be a homogeneous
  polynomial in two variables $x_1,x_2$ of degree at most $h\in\N$. If
  there exist $\bar\theta\in[0,2\pi)$ and an odd natural number $k$  such that $k>h$  and
\begin{equation}\label{eq:asshom}
Q\Big(\cos\big(\bar\theta+j\tfrac{2\pi}{k}\big),\sin\big(\bar\theta+j\tfrac{2\pi}{k}\big)\Big)
=0,\quad\text{for all }j=0,1,\dots,k-1,
\end{equation}
then $Q\equiv0$.
\end{Lemma}
\begin{proof}
  Up to a rotation, it is not restrictive to assume that
  $\bar\theta=0$.
If $x_1\neq0$, we can write $Q$ as 
\[
Q(x_1,x_2)=x_1^h\tilde Q\big(\tfrac{x_2}{x_1}\big),\quad \text{where }
\tilde Q(t)=\sum_{j=0}^h c_j t^{h-j}.
\]
Since $k$ is odd, we have that $\cos\big(j\tfrac{2\pi}{k}\big)\neq0$
for all $j=0,1,\dots,k-1$. 
Then, from assumption \eqref{eq:asshom} it follows that $\tilde
Q\big(\tan\big(j\tfrac{2\pi}{k}\big)\big)=0$ for all
$j=0,1,\dots,k-1$. 
Since $k$ is odd, we also have that
$\tan\big(j\tfrac{2\pi}{k}\big)\neq \tan\big(\ell\tfrac{2\pi}{k}\big)$
for all $j,l\in\{0,1,\dots,k-1\}$ with $j\neq\ell$. Hence $\tilde Q$
has $k$ distinct zeros. Since  $\tilde Q$ is a polynomial of degree at
most $h$ and $h<k$, from the Fundamental Theorem of Algebra we
conclude $\tilde Q\equiv 0$, i.e. $c_j=0$ for all
$j=0,1,\dots,k-1$. Hence $Q\equiv0$.
\end{proof}

\begin{proof}[Proof of Lemma \ref{l:taylor}]
Since the function
$a=(a_1,a_2)\mapsto\lambda_0-\lambda_a$ is $C^\infty$ in a neighborhood of $0$
(see \cite[Theorem 1.3]{BNNNT}), it admits a Taylor
expansion up to order $k$ of the form
\[
\lambda_0-\lambda_a=\sum_{j=1}^k P_j(a_1,a_2)+o(|a|^k),\quad\text{as
}|a|\to0,
\]
where, for every $j=1,\dots,k$, $P_j(a_1,a_2)$ is either identically
zero or a homogeneous
  polynomial in the two variables $a_1,a_2$ of degree $j$. From
  \cite[Theorem 1.2]{AF} (see also \eqref{eq:limlavve}) we have that,
for every $\ell<k$, 
\[
P_\ell\Big(\cos\big(\alpha_0+j\tfrac{2\pi}{k}\big),\sin\big(\alpha_0+j\tfrac{2\pi}{k}\big)\Big)=0,
\quad\text{for all }j=0,1,\dots,k-1,
\]
where $\alpha_0$ is as in \eqref{eq:alpha0}
(i.e. $\alpha_0+j\tfrac{2\pi}{k}$, with $j=0,1,\dots,k-1$, identify the
directions of the $k$ half-lines tangent to the nodal lines of the
eigenfunctions associated to $\lambda_0$). The conclusion follows directly from  Lemma
\ref{l:polinf}.
\end{proof}

From the expansion
\eqref{eq:tay}-~\eqref{eq:polinomio} in Lemma \ref{l:taylor} it
follows that
\begin{equation}\label{eq:estbr}
|\lambda_a-\lambda_0|=O(|a|^k)
\end{equation}
as $|a|\to0$ along any direction. Exploiting \eqref{eq:estbr} we can 
perform a sharp blow-up analysis prior to the estimate from above of
the eigenvalue variation $\lambda_0-\lambda_a$.

Let $\alpha\in[0,2\pi)$ and ${\mathbf
  p}=(\cos\alpha,\sin\alpha)$. Arguing as in \cite{AF} we can prove
that, for every $\delta\in(0,1/4)$, there exist $r_\delta,K_\delta>0$
such that, for all $R\geq K_\delta$,
\begin{equation}\label{eq:67}
\text{the family of functions }\big\{\tilde
 \varphi_a: a=|a|{\mathbf p}, |a|<\tfrac{r_\delta}{R}\big\}
 \text{ is bounded in $H^{1 ,{\mathbf p}}(D_R,\C)$}
\end{equation}
where 
\begin{equation}\label{def_blowuppate_normalizzate}
 \tilde \varphi_a (x) :=
 \dfrac{\varphi_a(|a|x)}{\sqrt{H_{a,\delta}}},
\end{equation}
and 
\[
H_{a,\delta}: = \dfrac1{K_\delta |a|} \int_{\partial D_{K_\delta |a|}} |\varphi_a|^2\,ds.
\]
Furthermore, from \cite[Estimates (113) and (114)]{AF} we have that 
\begin{equation}\label{eq:stima_sotto_radiceH}
H_{a,\delta} \geq C_\delta |a|^{k +
   2\delta},\quad\text{if }|a|<\frac{r_\delta}{K_\delta},
\end{equation}
for some $C_\delta>0$ independent of $a$,
and\begin{equation}\label{eq:stima_sopra_radiceH}
 H_{a,\delta}=O(|a|^{1-2\delta})\quad\text{ as }|a|\to0.
\end{equation}
For a precise proof of such estimates we refer the reader to \cite{AF}. 
Here we only mention that they proceed from suitable integrations of the monotonicity formula.  

 We observe that $\tilde \varphi_a$ weakly
solves
\begin{equation}\label{eq:3phiat}
  (i\nabla + A_{\mathbf p})^2 \tilde\varphi_a =|a|^2 \lambda_a
  \tilde\varphi_a,  \quad\text{in }\tfrac1{|a|}\Omega
  =\{x\in\R^2:|a|x\in\Omega\},
\end{equation}
and 
\begin{equation}\label{eq:normaliz}
  \frac{1}{K_\delta}\int_{\partial D_{K_\delta}}|\tilde\varphi_a|^2\,ds=1.
\end{equation}

Let $R>2$.  For $|a|$ sufficiently small we define the functions
$v_{j,R,a}$ as follows:
\[
 v_{j,R,a}= 
 \begin{cases}
  v_{j,R,a}^{ext}, &\text{in }\Omega \setminus D_{R|a|},\\
  v_{j,R,a}^{int}, &\text{in } D_{R|a|},
 \end{cases}
\quad j=1,\ldots,n_0,
\]
where 
\begin{equation*}
  v_{j,R,a}^{ext} := e^{\frac{i}{2}(\theta_0^a - \theta_a)}
  \varphi_j^a\quad \text{in }
  \Omega \setminus D_{R|a|}, 
\end{equation*}
with $\varphi_j^a$ as in \eqref{eq_eigenfunction}--\eqref{eq:89} and
$\theta_a,\theta_0^a$ as in \eqref{eq:4th} (notice that
$e^{\frac{i}{2}(\theta_0^a - \theta_a)}$ is smooth in $\Omega \setminus D_{R|a|}$),  so that it solves
\begin{equation*}
 \begin{cases}
   (i\nabla +A_0)^2 v_{j,R,a}^{ext} = \lambda_j^a v_{j,R,a}^{ext}, &\text{in }\Omega \setminus D_{R|a|},\\
   v_{j,R,a}^{ext} = e^{\frac{i}{2}(\theta_0^a - \theta_a)} \varphi_j^a
   &\text{on }\partial (\Omega \setminus D_{R|a|}),
 \end{cases}
\end{equation*}
whereas $v_{j,R,a}^{int}$ is the unique solution to the minimization problem 
\begin{multline*}
  \int_{D_{R|a|}} |(i\nabla +A_0)
  v_{j,R,a}^{int}(x)|^2\,dx\\
  = \min\left\{ \int_{D_{R|a|}} |(i\nabla +A_0)u(x)|^2\,dx:\, u\in
    H^{1,0}(D_{R|a|},\C), \ u= e^{\frac{i}{2}(\theta_0^a- \theta_a)}
    \varphi_j^a \text{ on }\partial D_{R|a|} \right\},
\end{multline*}
so that it solves
\begin{equation*}
 \begin{cases}
  (i\nabla +A_0)^2 v_{j,R,a}^{int} = 0, &\text{in }D_{R|a|},\\
  v_{j,R,a}^{int} = e^{\frac{i}{2}(\theta_0^a - \theta_a)} \varphi_j^a, &\text{on }\partial D_{R|a|}.
 \end{cases}
\end{equation*}
It is easy to verify that 
$\mathop{\rm dim}\big(\mathop{\rm span} \{v_{1,R,a},\ldots,v_{n_0,R,a}\}\big)=n_0$.

For all  $R> 2$ and  $a=|a|{\mathbf p}\in\Omega$ with
  $|a|$ small, we define 
\begin{equation}\label{eq:zar}
Z_a^R(x):=\frac{v_{n_0,R,a}^{int} (|a|x)}{\sqrt{H_{a,\delta}}}.
\end{equation}
Arguing as in \cite[Lemma 6.2]{AF} we can prove that, as a consequence
of \eqref{eq:67} and the Dirichlet principle, 
\begin{equation}\label{eq:67_zar}
\text{the family of functions }\big\{Z_a^R: a=|a|{\mathbf p}, |a|<\tfrac{r_\delta}{R}\big\}
 \text{ is bounded in $H^{1 ,0}(D_R,\C)$}.
\end{equation}

\begin{Theorem}\label{stima_teo_inversione}
For every  $R>2$,
\begin{equation*}
\|v_{n_0,R,a} -
\varphi_0\|_{H^{1,0}_0(\Omega,\C)}=O\Big(\sqrt{H_{a,\delta}}\Big)
\quad \text{as }a=|a|{\mathbf p}\to0.
\end{equation*}
\end{Theorem}
\begin{proof}
Let $R>2$. We first notice that $v_{n_0,R,a}\to \varphi_0$ in
$H^{1,0}_0(\Omega,\C)$ as $|a|\to 0^+$. Indeed
\begin{align*}
  & \int_{\Omega}\big|(i\nabla+A_0)(v_{n_0,R,a}-\varphi_0)\big|^2\,dx=
  \int_{\Omega}| e^{\frac{i}{2}(\theta_0^a - \theta_a)}
(i\nabla+A_a)\varphi_a-
  (i\nabla+A_0)\varphi_0|^2\,dx \\
  &\notag\qquad +
  \int_{D_{R}}\Big|\sqrt{H_{a,\delta}} (i\nabla+A_0)Z_a^{R}-
|a|^{k/2}
(i\nabla+A_0) W_a
  \Big)\Big|^2\,dx\\
  &\notag\qquad -
  \int_{D_{R}}\Big| \sqrt{H_{a,\delta}}e^{\frac{i}{2}(\theta_0^{\mathbf p}-
    \theta_{\mathbf p})} (i\nabla+A_{\mathbf
    p})\tilde\varphi_a-|a|^{k/2}
  (i\nabla+A_0)W_a \Big)\Big|^2\,dx=o(1)
\end{align*} 
in view of \eqref{eq:congrad2}, \eqref{eq:67}, \eqref{eq:67_zar},
\eqref{eq:vkext_la} and
\eqref{eq:stima_sopra_radiceH}.

From \cite[Lemma 7.1]{AF} the function
\begin{align*}
  F: \C \times H^{1,0}_{0}(\Omega,\C) &\longrightarrow
  \R \times \R \times (H^{1,0}_{0,\R}(\Omega,\C))^\star\\
  (\lambda,\varphi) &\longmapsto \Big( {\textstyle{
      \|\varphi\|_{H^{1,0}_0(\Omega,\C)}^2 -\lambda_0,\
      \mathfrak{Im}\big(\int_{\Omega}
      \varphi\overline{\varphi_0}\,dx\big), \ (i\nabla +A_0)^2
      \varphi-\lambda \varphi}}\Big)
\end{align*}
 is Frech\'{e}t-differentiable at $(\lambda_0,\varphi_0)$ and its
  Frech\'{e}t-differential $dF(\lambda_0,\varphi_0)$ 
is invertible. In the above definition, $(H^{1,0}_{0,\R}(\Omega,\C))^\star$ is the real dual space of
  $H^{1,0}_{0,\R}(\Omega,\C)=H^{1,0}_{0}(\Omega,\C)$, which is here meant as a
vector space over $\R$ endowed with the norm
\[
\|u\|_{H^{1,0}_0(\Omega,\C)}=\bigg(
\int_{\Omega}\big|(i\nabla +A_0)u\big|^2dx\bigg)^{\!\!1/2}.
\]
Therefore 
\begin{multline*}
  |\lambda_{a} - \lambda_0| + \|v_{n_0,R,a} -
  \varphi_0\|_{H^{1,0}_0(\Omega,\C)}\\
  \leq\|(dF(\lambda_0,\varphi_0))^{-1}\|_{ \mathcal L( \R\times \R
    \times (H^{1,0}_{0,\R}(\Omega,\C))^\star,\C \times
    H^{1,0}_{0}(\Omega,\C))} \| F(\lambda_a,v_{n_0,R,a})\|_{
    \R\times\R \times (H^{1,0}_{0,\R}(\Omega))^\star} (1+o(1))
\end{multline*}
as $|a|\to 0^+$.
To prove the theorem it is then enough to estimate the norm of
\begin{align*}
  F&(\lambda_a,v_{n_0,R,a}) =\left( \alpha_a, \beta_a, w_a\right)
  \\
  \notag& = \left( \|v_{n_0,R,a}\|_{H^{1,0}_0(\Omega,\C)}^2
    -\lambda_0,
    \mathfrak{Im}\left({\textstyle{\int_{\Omega}v_{n_0,R,a}\overline{\varphi_0}\,dx}}\right),
    (i\nabla+A_0)^2 v_{n_0,R,a} - \lambda_a v_{n_0,R,a} \right)
\end{align*}
in $\R\times\R \times (H^{1,0}_{0,\R}(\Omega))^\star$.
The estimates of $\beta_a$ and $w_a$ can be performed as in
\cite[Proof of Theorem 7.2]{AF} obtaining that 
\[
\beta_a =o\big(\sqrt{H_{a,\delta}}\big)\quad\text{and}
\quad
\|w_a\|_{(H^{1,0}_{0,\R}(\Omega,\C))^\star}=O\big(\sqrt{H_{a,\delta}}\big),
\]
as $a=|a|{\mathbf p}\to0$.
As far as $\alpha_a$ is concerned, 
differently from \cite{AF}, the estimate  of
\cite[Proposition 6.10]{AF} which, in the case $\alpha=0$, implied that
$|\lambda_a-\lambda_0|=O(H_{a,\delta})$, is not available after
preliminary estimates of the Rayleigh quotient for
generic values of $\alpha$ since $(\xi_{\mathbf p} (1)-\sqrt\pi)$ can have
any sign. This difficulty can be overcome by observing that
\eqref{eq:stima_sotto_radiceH} and \eqref{eq:estbr} imply that
$|\lambda_a-\lambda_0|=O(|a|^{\frac k2-\delta}\sqrt{H_{a,\delta}})$
and then 
\[
|\lambda_a-\lambda_0|=o(\sqrt{H_{a,\delta}}),
\]
as $a=|a|{\mathbf p}\to0$.
Then, from \eqref{eq:67} and \eqref{eq:67_zar}, we obtain that 
\begin{align*}
 \alpha_a 
 &= \left(
 \int_{ D_{R|a|}} |(i\nabla+A_0)v_{n_0,R,a}^{int}|^2 \,dx-
 \int_{D_{R|a|}} |(i\nabla+A_a)\varphi_a|^2\,dx 
 \right) +(\lambda_a-\lambda_0)\\
 &=  H_{a,\delta}
\left(
 \int_{ D_{R}} |(i\nabla+A_0)Z_a^R|^2 \,dx-
 \int_{D_{R}} |(i\nabla+A_{\mathbf p})\tilde\varphi_a|^2\,dx 
 \right) +(\lambda_a-\lambda_0)
\\
 &= o(\sqrt{H_{a,\delta}}),
\end{align*}
as $a=|a|{\mathbf p}\to0$, thus concluding the proof.
\end{proof}

From Theorem \ref{stima_teo_inversione} and scaling, it follows that, letting
$\alpha\in[0,2\pi)$, 
${\mathbf
  p}=(\cos\alpha,\sin\alpha)$, and $R>1$, 
\begin{equation}\label{eq:stscal}
\int_{\big(\frac{1}{|a|}\Omega\big)\setminus D_{R}}\bigg|(i\nabla+A_{\mathbf p})
 \Big(\tilde \varphi_a(x) - e^{\frac i2(\theta_{\mathbf p}-\theta_0^{\mathbf p})}
\tfrac{|a|^{k/2}}{\sqrt{H_{a,\delta}}}W_a\Big)\bigg|^2dx=O(1),
\quad\text{as }a=|a|{\mathbf p}\to0.  
\end{equation}

\begin{Theorem}\label{t:blowup}
  For $\alpha\in[0,2\pi)$, ${\mathbf p}=(\cos\alpha,\sin\alpha)$ and
  $a=|a|{\mathbf p}\in\Omega$, let $\varphi_a\in
  H^{1,a}_{0}(\Omega,\C)$ solve (\ref{eq:equation_a}-\ref{eq:6}) and
  $\varphi_0\in H^{1,0}_{0}(\Omega,\C)$ be a solution to
  \eqref{eq:equation_lambda0} satisfying \eqref{eq:1},
  \eqref{eq:37}, and \eqref{eq:54}.
 Let $\tilde\varphi_a$ and
  $K_\delta$ be as in \eqref{def_blowuppate_normalizzate},
  $\beta_2$ as in \eqref{eq:131}, and $\Psi_{\mathbf p}$
  be as in Proposition \ref{prop_Psi}.  Then
\begin{equation}\label{eq:58}
  \lim_{|a|\to 0^+}\frac{|a|^{k/2}}{\sqrt{H_{a,\delta}}}=
  \frac{\sqrt\pi}{|\beta_2|}
  \sqrt{\frac{K_\delta}{\int_{\partial D_{K_\delta}}|\Psi_{\mathbf p}|^2ds}}
\end{equation}
and
\begin{equation}\label{eq:59}
  \tilde \varphi_{a}\to \frac{\beta_2}{|\beta_2|}
  \sqrt{\frac{K_\delta}{\int_{\partial D_{K_\delta}}|\Psi_{\mathbf p}|^2ds}} \ \Psi_{\mathbf p} \quad\text{as
  }a=|a|{\mathbf p}\to0,
\end{equation}
in $H^{1 ,{\mathbf p}}(D_R,\C)$ for every $R>1$, almost everywhere and
in $C^{2}_{\rm loc}(\R^2\setminus\{{\mathbf p}\},\C)$.
\end{Theorem}
\begin{proof}
{\bf Step 1.} We first prove that for every sequence
$a_n=|a_n|\mathbf p$ with $|a_n|\to 0$, there
exist $\tilde \Phi\in H^{1 ,{\mathbf p}}_{\rm loc}(\R^2,\C)$, $\tilde
\Phi\not\equiv 0$, and a
subsequence $a_{n_\ell}$ such that $\tilde
\varphi_{a_{n_\ell}}\to\tilde \Phi$ in $H^{1
  ,{\mathbf p}}(D_R,\C)$ for every $R>1$, almost everywhere and
in $C^{2}_{\rm loc}(\R^2\setminus\{{\mathbf p}\},\C)$ and 
$\tilde \Phi$ weakly solves 
\begin{equation}\label{eq:tildePhi}
(i\nabla +A_{\mathbf p})^2 \tilde\Phi =0, \quad \text{in }\R^2.
\end{equation}
To prove it, we observe that from \eqref{eq:67} it follows that, for every sequence
$a_n=|a_n|\mathbf p$ with $|a_n|\to 0$, by a diagonal process there
exists $\tilde \Phi\in H^{1 ,{\mathbf p}}_{\rm loc}(\R^2,\C)$, and a
subsequence $a_{n_\ell}$ such that $\tilde
\varphi_{a_{n_\ell}}\rightharpoonup \tilde \Phi$ weakly in $H^{1
  ,{\mathbf p}}(D_R,\C)$ for every $R>1$ and almost everywhere.
$\tilde \Phi\not\equiv 0$ since $\frac1{K_\delta}\int_{\partial
  D_{K_\delta}}|\tilde\Phi|^2\,ds=1$ thanks to \eqref{eq:normaliz} and
the compactness of the trace embedding $H^{1,{\mathbf
    p}}(D_{K_{\delta}},\C)\hookrightarrow L^2(\partial
D_{K_{\delta}},\C)$.

Passing to the limit in \eqref{eq:3phiat}, we have that
$\tilde \Phi$ weakly solves  \eqref{eq:tildePhi},
whereas, arguing as in the proof of \cite[Theorem 8.1]{AF},  we can
prove that the convergence of  the subsequence $\tilde \varphi_{a_{n_\ell}}$ to
$\tilde \Phi$ is actually strong in $H^{1 ,{\mathbf p}}(D_R,\C)$ for
every $R>1$. The convergence in $C^{2}_{\rm
  loc}(\R^2\setminus\{{\mathbf p}\},\C)$ follows easily from classical elliptic estimates.

\medskip\noindent
{\bf Step 2.} We claim that, for every sequence
$a_n=|a_n|{\mathbf p}$ with $|a_n|\to 0$, there
exists a
subsequence $a_{n_\ell}$ such that 
\[
 \lim_{\ell\to
   +\infty}\frac{|a_{n_\ell}|^{k/2}}{\sqrt{H_{a_{n_\ell},\delta}}}
\quad \text{is finite and strictly positive}.
\]
To prove the claim, we argue by contradiction, assuming that 
\begin{enumerate}[\rm(i)]
\item either there
exists a sequence $a_n=|a_n|{\mathbf p}$ with $|a_n|\to 0$ such that 
$\lim_{n\to
   +\infty}\frac{|a_{n}|^{k/2}}{\sqrt{H_{a_{n},\delta}}}=0$\\
\item 
 or  there
exists a sequence $a_n=|a_n|{\mathbf p}$ with $|a_n|\to 0$ such that 
$\lim_{n\to
   +\infty}\frac{|a_{n}|^{k/2}}{\sqrt{H_{a_{n},\delta}}}=+\infty$.
\end{enumerate}
If (i) holds, then, by step 1, along a subsequence, 
 $\tilde
\varphi_{a_{n_\ell}}\to\tilde \Phi$ in $H^{1
  ,{\mathbf p}}(D_R,\C)$ for every $R>1$, for some $\tilde
\Phi\not\equiv 0$ weakly solving \eqref{eq:tildePhi}. Then from 
 \eqref{eq:vkext_la}, passing to the limit in \eqref{eq:stscal} we
 would obtain that 
\begin{equation*}
 \int_{\R^2\setminus D_{R}}|(i\nabla+A_{\mathbf p})\tilde\Phi(x)|^2dx<+\infty,
\end{equation*} 
contradicting the fact that $\tilde \Phi\not\equiv 0$ is a non trivial
weak solution to \eqref{eq:tildePhi} (and so cannot have finite energy
otherwise by testing the equation we would get that $\tilde
\Phi\equiv 0$, see \cite[Proof of Proposition 4.3]{AF}). Hence case
(i) cannot occur.

If (ii) holds, then from  \eqref{eq:stscal} we would have, for all $R>2$ 
\begin{equation*}
\frac{|a|^{k}}{H_{a,\delta}}\int_{D_{2R}\setminus D_{R}}\bigg|(i\nabla+A_{\mathbf p})
 \Big(
\tfrac{\sqrt{H_{a,\delta}}}{|a|^{k/2}}
\tilde \varphi_a(x) - e^{\frac i2(\theta_{\mathbf p}-\theta_0^{\mathbf p})}
W_a\Big)\bigg|^2dx=O(1),
\quad\text{as }a=|a|{\mathbf p}\to0,  
\end{equation*}
and hence, in view of \eqref{eq:vkext_la} and \eqref{eq:67}, passing
to the limit along the sequence we would obtain that 
\begin{align*}
&\frac{|a_n|^{k}}{H_{a_n,\delta}}\bigg(\int_{D_{2R}\setminus D_{R}}\bigg|(i\nabla+A_{\mathbf p})
 \Big(
e^{\frac i2(\theta_{\mathbf p}-\theta_0^{\mathbf p})}
\beta e^{\frac i2\theta_0}\psi_k\Big)\bigg|^2dx+o(1)\bigg)\\
&\quad
=\frac{|a_n|^{k}}{H_{a_n,\delta}}\bigg(|\beta|^2\int_{D_{2R}\setminus
  D_{R}}|\nabla \psi_k|^2dx+o(1)\bigg)=O(1),
\quad\text{as }n\to+\infty,  
\end{align*}
which is not possibile 
if $\lim_{\ell\to
   +\infty}\frac{|a_{n_\ell}|^{k/2}}{\sqrt{H_{a_{n_\ell},\delta}}}=+\infty$
 as in case (ii), since $\int_{D_{2R}\setminus
  D_{R}}|\nabla \psi_k|^2dx>0$. Hence also case
(ii) cannot occur and the claim of step 2 is proved.

\medskip\noindent {\bf Step 3.} From steps 1 and 2, it follows that,
for every sequence $a_n=(|a_n|,0)=|a_n|\mathbf p$ with $|a_n|\to 0$,
there exist $c\in (0,+\infty)$, $\tilde \Phi\in H^{1 ,{\mathbf
    e}}_{\rm loc}(\R^2,\C)$ weakly solving \eqref{eq:tildePhi},
$\tilde \Phi\not\equiv0$, and a subsequence $a_{n_\ell}$ such that
$\lim_{\ell\to+\infty}\frac{|a_{n_\ell}|^{k/2}}{\sqrt{H_{a_{n_\ell},\delta}}}=c$
and $\tilde \varphi_{a_{n_\ell}} \to\tilde \Phi$ in $H^{1 ,{\mathbf
    p}}(D_R,\C)$ for every $R>1$ and in $C^{2}_{\rm
  loc}(\R^2\setminus\{{\mathbf p}\},\C)$.  Passing to the limit along
$a_{n_\ell}$ in \eqref{eq:stscal} and recalling \eqref{eq:vkext_la},
we obtain that, for every $R>2$,
\begin{equation*}
 \int_{\R^2\setminus D_{R}}\bigg|(i\nabla+A_{\mathbf p})
 \Big(\tilde\Phi(x) - c\beta 
e^{\frac i2(\theta_{\mathbf p}-\theta_0^{\mathbf p})} e^{\frac i2
  \theta_0}\psi_k\Big)\bigg|^2dx<+\infty,
\end{equation*}
where $\beta$ is defined in \eqref{eq:beta}. 
Hence from Proposition \ref{prop_Psi} we conclude that necessarily
\begin{equation}\label{eq:57}
\tilde\Phi=c\beta\Psi_{\mathbf p}.
\end{equation}
Since $\frac1{K_\delta}\int_{\partial
  D_{K_\delta}}|\tilde\Phi|^2\,ds=1$, from \eqref{eq:57} and the fact that $c$ is a positive
real number, it follows that 
$c=
\frac{1}{|\beta|}
\big(\frac{K_\delta}{\int_{\partial D_{K_\delta}}|\Psi_{\mathbf p}|^2ds}\big)^{1/2}$.
Hence we have that 
\[
\tilde \varphi_{a_{n_\ell}}\to 
\frac{\beta}{|\beta|}
\sqrt{\tfrac{K_\delta}{\int_{\partial D_{K_\delta}}|\Psi_{\mathbf p}|^2ds}}
\Psi_{\mathbf p}
\quad\text{ in }H^{1 ,{\mathbf p}}(D_R,\C)\text{  
 for
every $R>1$ and in }C^{2}_{\rm
  loc}(\R^2\setminus\{{\mathbf p}\},\C),
\] 
and 
\[
\frac{|a_{n_\ell}|^{k/2}}{\sqrt{H_{a_{n_\ell},\delta}} }\to \frac{1}{|\beta|}
\sqrt{\frac{K_\delta}{\int_{\partial D_{K_\delta}}|\Psi_{\mathbf p}|^2ds}}.
\]
Since the above limits depend neither on the sequence $\{a_{n}\}_n$ nor on the
  subsequence $\{a_{n_\ell}\}_\ell$, we conclude that the above
  convergences hold as $|a|\to 0^+$, thus proving \eqref{eq:58} and \eqref{eq:59}.
 \end{proof}

 \begin{remark}\label{rem:blow}
   Combining   \eqref{eq:58} and \eqref{eq:59} we deduce that 
\begin{equation*}
 \frac{\varphi_a(|a|x)}{|a|^{k/2}}\to \frac{\beta_2}{\sqrt\pi} \Psi_{\mathbf p} \quad\text{as
  }a=|a|{\mathbf p}\to0,
\end{equation*}
in $H^{1 ,{\mathbf p}}(D_R,\C)$ for every $R>1$ and
in $C^{2}_{\rm loc}(\R^2\setminus\{{\mathbf p}\},\C)$.
Furthermore, arguing as in \cite[Lemma 8.3]{AF}, from Theorem
\ref{t:blowup} we can deduce that, letting $Z_a^R$ as is
\eqref{eq:zar},  
\begin{equation*}
  Z_a^R\to \frac{\beta_2}{|\beta_2|}
  \sqrt{\frac{K_\delta}{\int_{\partial D_{K_\delta}}|\Psi_{\mathbf p}|^2ds}} \ z_R \quad\text{as
  }a=|a|{\mathbf p}\to0,
\end{equation*}
in $H^{1 ,0}(D_R,\C)$ for every $R>2$, where $z_R$ is the unique
solution to 
\begin{equation*}
 \begin{cases}
  (i\nabla +A_0)^2 z_R = 0, &\text{in }D_{R},\\
  z_R = e^{\frac{i}{2}(\theta_{0}^{\mathbf p}-\theta_{\mathbf p})}\Psi_{\mathbf p}, &\text{on }\partial D_{R}.
 \end{cases}
\end{equation*}
\end{remark}
Thanks to the convergences of blow-up sequences established in Theorem
\ref{t:blowup} and Remark \ref{rem:blow}, we can now follow closely
the arguments of \cite[Subsection 6.1, Lemma 9.1]{AF} thus obtaining
the following upper bound for the difference $\lambda_0-\lambda_a$.

\begin{Lemma}\label{l:stima_Lambda0_sopra}
 For $\alpha\in[0,2\pi)$ and $a=|a|(\cos\alpha,\sin\alpha)\in\Omega$, let 
$\lambda_a\in\R$ and $\varphi_a\in
H^{1,a}_{0}(\Omega,\C)$ solve (\ref{eq:equation_a}-\ref{eq:6})
and $\lambda_0\in\R$ and $\varphi_0\in
H^{1,0}_{0}(\Omega,\C)$ solve
\eqref{eq:equation_lambda0}.
If \eqref{eq:1} and
\eqref{eq:37} hold and \eqref{eq:54} is satisfied, then, for
$a=|a|(\cos\alpha,\sin\alpha)$ and ${\mathbf p}=(\cos\alpha,\sin\alpha)$,
\[
\limsup_{|a|\to 0}\frac{\lambda_0-\lambda_a}{|a|^k}\leq |\beta|^2 \ k
\sqrt\pi (\xi_{\mathbf p} (1)-\sqrt\pi ),
\]
with $\beta$ as in \eqref{eq:beta} and $\xi_{\mathbf p} (r)$ as defined in \eqref{eq:def_xi}.
\end{Lemma}

\bigskip

Collecting \eqref{eq:3.1} and Lemma \ref{l:stima_Lambda0_sopra}
we can state the following result.

\begin{Proposition}\label{p:primopasso}
 For $\alpha\in[0,2\pi)$ and $a=|a|(\cos\alpha,\sin\alpha)\in\Omega$, let 
$\varphi_a\in
H^{1,a}_{0}(\Omega,\C)$ and $\lambda_a\in\R$ solve (\ref{eq:equation_a}-\ref{eq:6})
and $\lambda_0\in\R$ and $\varphi_0\in
H^{1,0}_{0}(\Omega,\C)$ solve
\eqref{eq:equation_lambda0}.
If \eqref{eq:1} and
\eqref{eq:37} hold and \eqref{eq:54} is satisfied, then, for
$a=|a|(\cos\alpha,\sin\alpha)$,
\[
\lim_{|a|\to 0}\frac{\lambda_0-\lambda_a}{|a|^k}= |\beta|^2 \ k \sqrt\pi f(\alpha),
\]
where
\begin{equation}\label{eq:f(alpha)}
  f:[0,2\pi)\to\R,\quad f(\alpha)= (\xi_{\mathbf p} (1)-\sqrt\pi
  ),\quad
  {\mathbf p}=(\cos\alpha,\sin\alpha),
\end{equation}
with $\beta$ as in \eqref{eq:beta} and $\xi_{\mathbf p}(r)$ as defined in \eqref{eq:def_xi}.
\end{Proposition}

\section{Properties of $f(\alpha)$}\label{sec:properties-falpha} 

To prove our main result, we are going to investigate two suitable symmetry properties 
of the function $f(\alpha)$. 
Let us define two transformations $\mathcal R_1, \mathcal R_2$ acting on a general point 
\[
x=(x_1,x_2)=(r\cos t, r\sin t), \quad r>0,\ t\in[0,2\pi), 
\]
as 
\begin{equation}
\label{eq:trasformazioni1}
 \mathcal R_1 (x)= \mathcal R_1 (x_1,x_2) = 
{M}_k
 \begin{pmatrix}
   x_1\\
x_2
 \end{pmatrix},\quad
{M}_k=
 \begin{pmatrix}
   \cos\frac{2\pi}{k}&   -\sin\frac{2\pi}{k}\\[3pt]
   \sin\frac{2\pi}{k}&   \cos\frac{2\pi}{k}
 \end{pmatrix}
\end{equation}
i.e.
\[
\mathcal R_1 (r\cos t, r\sin t) = \Big(r\cos (t+\tfrac{2\pi}k), r\sin (t+\tfrac{2\pi}k)\Big),
\]
and 
\begin{equation}
\label{eq:trasformazioni2}
 \mathcal R_2 (x)= \mathcal R_2 (x_1,x_2) = (x_1,-x_2),
\end{equation}
i.e.
\[
\mathcal R_2 (r\cos t, r\sin t) = (r\cos (2\pi-t), r\sin (2\pi-t)),
\]
The transformation $\mathcal R_1$ is a rotation of $\frac{2\pi}k$ and
${\mathcal R}_2$ is a reflexion through the $x_1$-axis.

We would like to study how the coefficient
$\xi_{\mathbf p}(1)$ (see \eqref{eq:def_xi}) changes when the above trasformations
act on ${\mathbf p}$. In particular, we are
going  to prove that
such a quantity $\xi_{\mathbf p}(1)$ is invariant under the transformations
$\mathcal R_1,\mathcal R_2$.

In order to obtain such an invariance, we first study the
relation between the limit profiles $\Psi_{\mathbf p}(\mathcal R_j (x))$
and $\Psi_{\mathcal R_j^{-1}({\mathbf p})}(x)$, $j=1,2$.

\begin{Lemma}\label{l:relazionePsi}
 For  ${\mathbf p}=(\cos\alpha,\sin\alpha)$, $\alpha\in[0,2\pi)$, 
let $\Psi_{\mathbf p}$ be the limit profile introduced in Proposition
\ref{prop_Psi} and let $\mathcal R_1,\mathcal R_2$ be the transformations
introduced in \eqref{eq:trasformazioni1} and \eqref{eq:trasformazioni2}.
 Then 
 \begin{equation}\label{eq:relazionePsi1}
\Psi_{\mathcal R_1^{-1}({\mathbf p})}=-e^{-i\frac\pi k }
\big(\Psi_{\mathbf p}\circ \mathcal R_1\big)
\end{equation}
and 
 \begin{equation}\label{eq:relazionePsi2}
\Psi_{\mathcal R_2({\mathbf p})}=-e^{i\theta_{\mathcal R_2({\mathbf
      p})}}
\big(\Psi_{\mathbf p}\circ \mathcal R_2\big).
\end{equation}
\end{Lemma}
\begin{proof}
In order to prove \eqref{eq:relazionePsi1}, we observe that, by direct
calculations,
\begin{align}\label{eq:2}
&  \big(A_{\mathbf p}\circ{\mathcal R}_1\big)(x)=A_{\mathcal
    R_1^{-1}({\mathbf p})}(x)M_k^{-1},\\
\label{eq:3}
&e^{\frac i2(\theta_0\circ\mathcal R_1)}(\psi_k\circ\mathcal
R_1)=-e^{i\frac\pi k}e^{\frac i2\theta_0}\psi_k,\\
\label{eq:4}
&e^{\frac i2(\theta_0\circ\mathcal R_1)(x)}\nabla\psi_k(\mathcal
R_1(x))=-e^{i\frac\pi k}e^{\frac i2\theta_0 (x)}\nabla
\psi_k(x)M_k^{-1}.
\end{align}
Furthermore 
\[\theta_{\mathbf p}(\mathcal R_1(x))=
\begin{cases}
  \theta_{\mathcal R_1^{-1}({\mathbf p})}(x)+\frac{2\pi}k,\quad
  \text{if }\alpha\in \big[\frac{2\pi}k,2\pi\big),\\
  \theta_{\mathcal R_1^{-1}({\mathbf p})}(x)+\frac{2\pi}k-2\pi,\quad
  \text{if }\alpha\in \big[0,\frac{2\pi}k\big),
\end{cases}
\]
and
\[
\theta_0^{\mathbf p}(\mathcal R_1(x))=
\begin{cases}
  \theta_0^{\mathcal R_1^{-1}({\mathbf p})}(x)+\frac{2\pi}k,\quad
  \text{if }\alpha\in \big[\frac{2\pi}k,2\pi\big),\\
  \theta_0^{\mathcal R_1^{-1}({\mathbf p})}(x)+\frac{2\pi}k-2\pi,\quad
  \text{if }\alpha\in \big[0,\frac{2\pi}k\big),
\end{cases}
\]
so that 
\begin{equation}\label{eq:8}
  \theta_{\mathcal R_1^{-1}({\mathbf p})}-\theta_0^{\mathcal
    R_1^{-1}({\mathbf p})}=
\theta_{\mathbf p}\circ \mathcal R_1-\theta_0^{\mathbf p}\circ \mathcal R_1.
\end{equation}
Let us denote $\widetilde\Psi_{\mathbf p}(y)=\Psi_{\mathbf p}(\mathcal
R_1(y))$. By direct calculations we have that, since $\Psi_{\mathbf
  p}$ weakly  solves the equation $(i\nabla +A_{\mathbf p})^2\Psi_{\mathbf
  p}=0$, the function $\widetilde\Psi_{\mathbf p}$ solves 
$(i\nabla +(A_{\mathbf p}\circ\mathcal R_1)M_k)^2\widetilde\Psi_{\mathbf
  p}=0$ and hence, in view of \eqref{eq:2}, 
\begin{equation}\label{eq:12tras}
(i\nabla +
A_{\mathcal
    R_1^{-1}({\mathbf p})})^2\widetilde\Psi_{\mathbf
  p}=0,\quad\text{ in $\R^2$ in a
  weak $H^{1 ,\mathcal R_1^{-1}({\mathbf p})}$-sense}.
\end{equation}
Passing to the limit in \eqref{eq:stscal} and taking into account
\eqref{eq:vkext_la} and Theorem \ref{t:blowup}, we obtain that, for
all $R>1$, 
\begin{multline}\label{eq:11}
  \int_{\R^2\setminus D_{R}}\bigg|(i\nabla+A_{\mathbf p})
  \Big(\Psi_{\mathbf p} - e^{\frac i2(\theta_{\mathbf
      p}-\theta_0^{\mathbf p}+\theta_0)}\psi_k\Big)\bigg|^2dx\\
  =\int_{\R^2\setminus D_{R}}\bigg|(i\nabla+A_{\mathbf p})
  \Psi_{\mathbf p} - e^{\frac i2(\theta_{\mathbf p}-\theta_0^{\mathbf
      p}+\theta_0)}i\nabla\psi_k\bigg|^2dx <+\infty.
\end{multline}
By the change of variable $x=\mathcal R_1(y)$ in the above integral,
using \eqref{eq:2}, \eqref{eq:4}, and \eqref{eq:8} we obtain that 
\begin{gather}\label{eq:13}
  \int_{\R^2\setminus D_{R}}\bigg|(i\nabla+A_{\mathbf p})
  \Psi_{\mathbf p} (x)- e^{\frac i2(\theta_{\mathbf
      p}-\theta_0^{\mathbf
      p}+\theta_0)(x)}i\nabla\psi_k(x)\bigg|^2dx\\
  \notag= \int_{\R^2\setminus D_{R}}\bigg|(i\nabla+A_{\mathcal
    R_1^{-1}({\mathbf p})} ) \widetilde \Psi_{\mathbf p} (y)+e^{\frac
    ik\pi} e^{\frac i2\big(\theta_{\mathcal R_1^{-1}(\mathbf
      p)}-\theta_0^ {\mathcal R_1^{-1}(\mathbf
      p)}+\theta_0\big)(y)}i\nabla\psi_k(y)\bigg|^2dy\\
  \notag= \int_{\R^2\setminus D_{R}}\bigg|(i\nabla+A_{\mathcal
    R_1^{-1}({\mathbf p})} ) \big(-e^{-\frac ik\pi} \widetilde
  \Psi_{\mathbf p}\big) (y)- e^{\frac i2\big(\theta_{\mathcal
      R_1^{-1}(\mathbf p)}-\theta_0^ {\mathcal R_1^{-1}(\mathbf
      p)}+\theta_0\big)(y)}i\nabla\psi_k(y)\bigg|^2dy<+\infty.
\end{gather}
From \eqref{eq:12tras}, \eqref{eq:13} and Proposition
\ref{prop_Psi} we conclude that 
\[
-e^{-\frac ik\pi} \widetilde
  \Psi_{\mathbf p}=
\Psi_{\mathcal
    R_1^{-1}({\mathbf p})}
\]
thus proving \eqref{eq:relazionePsi1}.

To prove \eqref{eq:relazionePsi2}, we first observe that direct
calculations yield
\begin{align}\label{eq:2-2}
&  \big(A_{\mathcal R_2({\mathbf p})}\circ{\mathcal
  R}_2\big)M^{-1}=-A_{\mathbf p},\quad\text{where $M=
\begin{pmatrix}
  1&0\\
0&-1
\end{pmatrix}
$},\\
\label{eq:3-2}
&\psi_k\circ\mathcal
R_2=\psi_k,\quad \nabla\psi_k(\mathcal
R_2(x))=\nabla
\psi_k(x)M^{-1}.
\end{align}
Moreover
\begin{align*}
&\theta_{\mathbf p}(\mathcal R_2(x))=
\begin{cases}
  4\pi-\theta_{\mathcal R_2({\mathbf p})}(x),\quad
  \text{if }\theta_{\mathbf p}(\mathcal R_2(x))\in
  \big(\alpha,\alpha+2\pi\big),\\
  2\pi-\theta_{\mathcal R_2({\mathbf p})}(x),\quad
  \text{if }\theta_{\mathbf p}(\mathcal R_2(x))=\alpha,
\end{cases}\quad\text{if }\alpha\in(0,2\pi),\\
&\theta_{\mathbf p}(\mathcal R_2(x))=
\begin{cases}
  2\pi-\theta_{\mathcal R_2({\mathbf p})}(x)=2\pi-\theta_{{\mathbf p}}(x),&
  \text{if }\theta_{\mathbf p}(\mathcal R_2(x))\in
  \big(0,2\pi\big),\\
  -\theta_{\mathcal R_2({\mathbf p})}(x)=0,&
  \text{if }\theta_{\mathbf p}(\mathcal R_2(x))=0,
\end{cases}\quad\text{if }\alpha=0,
\end{align*}
and 
\begin{align*}
&\theta_0^{\mathbf p}(\mathcal R_2(x))=
\begin{cases}
  4\pi-\theta_0^{\mathcal R_2({\mathbf p})}(x),\quad
  \text{if }\theta_0^{\mathbf p}(\mathcal R_2(x))\in
  \big(\alpha,\alpha+2\pi\big),\\
  2\pi-\theta_0^{\mathcal R_2({\mathbf p})}(x),\quad
  \text{if }\theta_0^{\mathbf p}(\mathcal R_2(x))=\alpha,
\end{cases}\quad\text{if }\alpha\in(0,2\pi),\\
&\theta_0^{\mathbf p}(\mathcal R_2(x))=\theta_0(\mathcal R_2(x))=
\begin{cases}
2\pi-\theta_0(x),&
  \text{if }\theta_0(x)\in
  \big(0,2\pi\big),\\
  -\theta_0(x)=0,&
  \text{if }\theta_0(x)=0,
\end{cases}\quad\text{if }\alpha=0,
\end{align*}
so that 
\begin{equation}\label{eq:8-2}
\theta_0^{\mathcal
    R_2({\mathbf p})}-  \theta_{\mathcal R_2({\mathbf p})}=
\theta_{\mathbf p}\circ \mathcal R_2-\theta_0^{\mathbf p}\circ
\mathcal R_2,\quad \text{in }\R^2\setminus\{t{\mathbf p}:\, t\in[0,1]\},
\end{equation}
and 
\begin{equation}\label{eq:5-2}
  e^{\frac i2 \theta_0(\mathcal R_2 (y))}=
 - e^{-\frac i2 \theta_0(y)},\quad \text{in }\R^2\setminus \{(x_1,0): x_1\geq 0\}.
\end{equation}

Let us denote $\widehat\Psi_{\mathbf p}(y)=-e^{i\theta_{\mathcal
    R_2(\mathbf p)}}\Psi_{\mathbf p}(\mathcal
R_2(y))$. In view of \eqref{eq:2-2}, it is easy to verify that  $\widehat\Psi_{\mathbf p}$ solves 
\begin{equation}\label{eq:12tras-2}
(i\nabla +
A_{\mathcal
    R_2({\mathbf p})})^2\widehat\Psi_{\mathbf
  p}=0,\quad\text{ in $\R^2$ in a
  weak $H^{1 ,\mathcal R_2({\mathbf p})}$-sense}.
\end{equation}
By the change of variable $x=\mathcal R_2(y)$ in the integral \eqref{eq:11},
using \eqref{eq:2-2}, \eqref{eq:3-2}, \eqref{eq:8-2}, and \eqref{eq:5-2} and observing
that, by \eqref{eq:diff}, 
$e^{-i\big(\theta_0^{\mathcal R_2(\mathbf
      p)}-\theta_0\big)}\equiv1$ in $\R^2\setminus D_{R}$,

we obtain that 
\begin{gather}\label{eq:13-2}
  \int_{\R^2\setminus D_{R}}\bigg|(i\nabla+A_{\mathbf p})
  \Psi_{\mathbf p} (x)- e^{\frac i2(\theta_{\mathbf
      p}-\theta_0^{\mathbf
      p}+\theta_0)(x)}i\nabla\psi_k(x)\bigg|^2dx\\
  \notag=
  \int_{\R^2\setminus D_{R}}\bigg|(i\nabla+(A_{\mathbf p}\circ
  \mathcal R_2)M)
  (-\Psi_{\mathbf p}\circ\mathcal R_2)(y)- e^{-\frac
    i2\big(\theta_{\mathcal R_2(\mathbf
      p)}-\theta_0^{\mathcal R_2(\mathbf
      p)}+\theta_0\big) (y)}i\nabla\psi_k(y)\bigg|^2dy\\
  \notag=
  \int_{\R^2\setminus D_{R}}\bigg|(i\nabla-A_{\mathcal R_2({\mathbf p})})
  (-\Psi_{\mathbf p}\circ\mathcal R_2)(y)-
e^{-\frac
    i2\big(\theta_{\mathcal R_2(\mathbf
      p)}-\theta_0^{\mathcal R_2(\mathbf
      p)}+\theta_0\big) (y)}i\nabla\psi_k(y)\bigg|^2dy\\
  \notag=
  \int_{\R^2\setminus D_{R}}\bigg|e^{i\big(\theta_{\mathcal R_2(\mathbf
      p)}-\theta_0^{\mathcal R_2(\mathbf
      p)}+\theta_0\big)}(i\nabla-A_{\mathcal R_2({\mathbf p})})
  (-\Psi_{\mathbf p}\circ\mathcal R_2)-
e^{\frac
    i2\big(\theta_{\mathcal R_2(\mathbf
      p)}-\theta_0^{\mathcal R_2(\mathbf
      p)}+\theta_0\big) }i\nabla\psi_k\bigg|^2dy\\
  \notag=
  \int_{\R^2\setminus D_{R}}\bigg|e^{i\theta_{\mathcal R_2(\mathbf
      p)}}(i\nabla-A_{\mathcal R_2({\mathbf p})})
  (-\Psi_{\mathbf p}\circ\mathcal R_2)-
e^{\frac
    i2\big(\theta_{\mathcal R_2(\mathbf
      p)}-\theta_0^{\mathcal R_2(\mathbf
      p)}+\theta_0\big) }i\nabla\psi_k\bigg|^2dy\\
  \notag=
  \int_{\R^2\setminus D_{R}}\bigg|(i\nabla+A_{\mathcal R_2({\mathbf p})})
\widehat\Psi_{\mathbf p}-
e^{\frac
    i2\big(\theta_{\mathcal R_2(\mathbf
      p)}-\theta_0^{\mathcal R_2(\mathbf
      p)}+\theta_0\big) }i\nabla\psi_k\bigg|^2dy.
\end{gather}
From \eqref{eq:12tras-2}, \eqref{eq:13-2} and Proposition
\ref{prop_Psi} we conclude that 
\[
\widehat\Psi_{\mathbf p}=\Psi_{\mathcal R_2({\mathbf p})}
\]
thus proving \eqref{eq:relazionePsi2}.
\end{proof}

We are now in position to prove invariance properties of the function
${\mathbf p}\mapsto\xi_{\mathbf p}(1)$ under the transformations
\eqref{eq:trasformazioni1} and \eqref{eq:trasformazioni2}.
\begin{Lemma}\label{l:invarianzaxi}
Let $\mathcal R_1,\mathcal R_2$ be the transformations introduced in
 \eqref{eq:trasformazioni1}-\eqref{eq:trasformazioni2},
 $\alpha\in[0,2\pi)$, and  ${\mathbf p}=(\cos\alpha,\sin\alpha)$.
 Then 
 \begin{equation}
  \xi_{\mathcal R_1^{-1}({\mathbf p})} (1) = \xi_{\mathbf p} (1)\label{eq:inva1}
\end{equation}
and
\begin{equation}
  \xi_{\mathcal R_2({\mathbf p})} (1) = \xi_{\mathbf p} (1),\label{eq:inva2}
\end{equation}
where $\xi_{\mathbf p}$ is defined in \eqref{eq:def_xi}. 
\end{Lemma}
\begin{proof}
We first notice that, from 
\eqref{eq:9}, 
\begin{equation}\label{eq:7inv}
  \psi_2^k\Big(s+\frac{2\pi}{k}\Big)=-e^{i\frac{\pi}{k}}\psi_2^k(s),\quad\text{for
  all }s\in\R,
\end{equation}
and 
\begin{equation}\label{eq:8inv}
  \psi_2^k(2\pi-s)=-e^{-is}\psi_2^k(s),\quad\text{for
  all }s\in\R.
\end{equation}
By the change of variable $t=s+\frac{2\pi}{k}$ in the integral
defining $\xi_{\mathbf p}(1)$, from \eqref{eq:7inv}, \eqref{eq:8} and
\eqref{eq:relazionePsi1} we obtain that 
\begin{align*}
  \xi_{\mathbf p} (1)&=
\int_{0}^{2\pi} e^{-\frac i2 (\theta_{\mathbf p}-\theta_0^{\mathbf p})
(\cos t,\sin t)} \Psi_{\mathbf p} (\cos t,\sin t)
 \overline{\psi_2^k(t)}\,dt\\
&=-e^{-i\frac{\pi}{k}}\int_{0}^{2\pi} e^{-\frac i2 (\theta_{\mathbf
    p}-\theta_0^{\mathbf
    p})
(\mathcal R_1(\cos s,\sin s))} \Psi_{\mathbf p} (\mathcal R_1(\cos s,\sin s))
\overline{\psi_2^k(s)}\,ds\\
&=\int_{0}^{2\pi} e^{-\frac i2 \big(\theta_{\mathcal R_1^{-1}({\mathbf
    p})}-\theta_0^{\mathcal R_1^{-1}({\mathbf
    p})}\big)(\cos s,\sin s)} \Psi_{\mathcal R_1^{-1}({\mathbf p})} (\cos s,\sin s)
 \overline{\psi_2^k(s)}\,ds\\
&= \xi_{\mathcal R_1^{-1}({\mathbf p})} (1),
\end{align*}
thus proving \eqref{eq:inva1}.

By the change of variable $t=2\pi-s$ in the integral
defining $\xi_{\mathbf p}(1)$, from \eqref{eq:8inv}, \eqref{eq:8-2},
\eqref{eq:relazionePsi2}, and \eqref{eq:diff} we obtain that 
\begin{align*}
&  \xi_{\mathbf p} (1)=
\int_{0}^{2\pi} e^{-\frac i2 (\theta_{\mathbf p}-\theta_0^{\mathbf p})
(\cos t,\sin t)} \Psi_{\mathbf p} (\cos t,\sin t)
 \overline{\psi_2^k(t)}\,dt\\
&=-\int_{0}^{2\pi} e^{-\frac i2 (\theta_{\mathbf
    p}-\theta_0^{\mathbf
    p})
(\mathcal R_2(\cos s,\sin s))} \Psi_{\mathbf p} (\mathcal R_2(\cos
s,\sin s))
e^{i\theta_0(\cos s,\sin s)}
\overline{\psi_2^k(s)}\,ds\\
&=\int_{0}^{2\pi} e^{-\frac i2 \big(\theta_0^{\mathcal R_2({\mathbf
    p})}-\theta_{\mathcal R_2({\mathbf
    p})}\big)(\cos s,\sin s)}
e^{-i\theta_{\mathcal R_2({\mathbf p})}(\cos s,\sin s)}
 \Psi_{\mathcal R_2({\mathbf p})} (\cos s,\sin s)
e^{i\theta_0(\cos s,\sin s)} 
\overline{\psi_2^k(s)}\,ds\\
&=\int_{0}^{2\pi} e^{-\frac i2 \big(
\theta_{\mathcal R_2({\mathbf
    p})}-\theta_0^{\mathcal R_2({\mathbf
    p})}\big)(\cos s,\sin s)}
e^{-i\big(\theta_0^{\mathcal R_2({\mathbf p})}-\theta_0\big)(\cos s,\sin s)}
 \Psi_{\mathcal R_2({\mathbf p})} (\cos s,\sin s)
\overline{\psi_2^k(s)}\,ds\\
&=\int_{0}^{2\pi} e^{-\frac i2 \big(
\theta_{\mathcal R_2({\mathbf
    p})}-\theta_0^{\mathcal R_2({\mathbf
    p})}\big)(\cos s,\sin s)}
 \Psi_{\mathcal R_2({\mathbf p})} (\cos s,\sin s)
\overline{\psi_2^k(s)}\,ds\\
&= \xi_{\mathcal R_2({\mathbf p})} (1),
\end{align*}
thus proving \eqref{eq:inva2}.
\end{proof}

Let $f$ be the $2\pi$-periodic extension of the function introduced in
\eqref{eq:f(alpha)}, i.e. 
\begin{equation}\label{eq:fext}
f(\alpha)=\xi_{(\cos\alpha,\sin\alpha)}(1)-\sqrt\pi,\quad\text{for all
  $\alpha\in\R$},
\end{equation}
with $\xi_{\mathbf p}$ defined in \eqref{eq:def_xi}.

\begin{Corollary}\label{cor:per}
 Let $f$ be defined in \eqref{eq:fext}. Then 
\[
f(\alpha)=f\big(\alpha-\tfrac{2\pi}{k}\big)\quad\text{and}
\quad 
f(\alpha)=f(2\pi-\alpha)
\]
for all $\alpha\in\R$.
\end{Corollary}
\begin{proof}
It is a straightforward consequence of Lemma \ref{l:invarianzaxi}.
\end{proof}

\section{Proof of the main result}\label{sec:proof-main-result}

From Lemma \ref{l:taylor} and Proposition \ref{p:primopasso}, it
follows that, under assumption \eqref{eq:54}, the homogeneous
polynomial $P$ \eqref{eq:polinomio} of degree $k$ appearing in the
expansion \eqref{eq:tay} is given by
\begin{equation}\label{eq:poly}
  P(r\cos\alpha,r\sin\alpha)=r^k|\beta|^2 k\sqrt\pi f(\alpha),
  \quad r>0,\ \alpha\in\R,
\end{equation}
with $f$ as in \eqref{eq:fext}.
Furthermore, from Corollary \ref{cor:per} and \eqref{eq:poly}, we have
that the $2\pi$-periodic function 
\begin{equation}\label{eq:def_g}
g:\R\to\R,\quad
g(\alpha):=P(\cos\alpha,\sin\alpha)=
\sum_{j=0}^k c_j (\cos\alpha)^{k-j}
(\sin\alpha)^j
\end{equation}
satisfy the periodicity/symmetry conditions
\begin{equation}\label{eq:5persym}
  g(\alpha)=g\big(\alpha+\tfrac{2\pi}{k}\big)\quad\text{and}
\quad 
g(\alpha)=g(2\pi-\alpha),\quad\text{for all }\alpha\in\R.
\end{equation}
From \cite[Theorem 1.2]{AF}
we also know that 
\begin{equation}\label{eq:f(0)}
c_0=  g(0)=-4\frac{|\beta_2|^2}{\pi}{\mathfrak m}_k>0,
\end{equation}
with ${\mathfrak m}_k$ being as in \eqref{eq:Ik}--\eqref{eq:segno_mk}
and $\beta_2$ as in \eqref{eq:131}.

\begin{Lemma}\label{l:fattorizzazione}
Under the assumptions of Lemma \ref{l:taylor} and \eqref{eq:54}, let
$P$ be as in \eqref{eq:tay}-\eqref{eq:polinomio} and $g$ be defined
in \eqref{eq:def_g}. Then 
 \begin{equation*}
   g(\alpha)=\frac{c_0}{\prod_{\ell=1}^k
   \sin\big(\frac\pi{2k}(2\ell-1)\big)}
 \prod_{j=1}^k\sin\Big(\frac\pi{2k}(2j-1)-\alpha\Big),\quad\text{for all }\alpha\in\R.
 \end{equation*}
\end{Lemma}
\begin{proof}
From \eqref{eq:f(0)} and \eqref{eq:5persym}, we have that 
\begin{equation}\label{eq:5}
  g\Big(j\frac{2\pi}{k}\Big)>0\quad\text{for all }j=0,1,\dots,\frac{k-1}{2}.
\end{equation}
Moreover, from \eqref{eq:def_g} and oddness of $k$, we have that
\begin{equation}\label{eq:18}
g(\alpha+\pi)=-g(\alpha),\quad\text{for all }\alpha\in\R,
\end{equation}
 and hence from \eqref{eq:5persym} we deduce that 
\begin{align}\label{eq:7}
  g\bigg(\frac\pi k+j\frac{2\pi}{k}\bigg)
&= g\Big(\pi+\big(\tfrac\pi k-\pi+j\tfrac{2\pi}{k}\big)\Big)\\
\notag&= g\Big(\pi+\tfrac{2\pi}{k}\big(j+\tfrac{1-k}2\big)\Big)=g(\pi)=-g(0)<0
\quad\text{for all }j=0,1,\dots,\frac{k-1}{2}.
\end{align}
From \eqref{eq:5} and \eqref{eq:7} 
we infer that 
 $g$ has at least $k$ distinct zeros
 $\theta_1,\theta_2,\dots,\theta_k$ in $(0,\pi)$ such that 
 \begin{equation*}
   (j-1)\frac\pi k< \theta_j<j\frac\pi k,\quad\text{for all }j=1,\dots,k.
 \end{equation*}
In view of this fact, we aim at factorizing the function $g$. 
For every $\alpha\in\R\setminus\{\ell\pi:\,\ell\in\Z\}$ we have that 
\begin{equation}\label{eq:10}
  g(\alpha)=(\sin\alpha)^k \widetilde P(\cot\alpha) 
\end{equation}
where 
\[
 \widetilde P(t)=
\sum_{j=0}^k c_j t^{k-j}.
\]
From \eqref{eq:f(0)} the $1$-variable polynomial $\widetilde P$ has
degree $k$. Furthermore, by \eqref{eq:10}, $\cot\theta_1,\dots,\cot\theta_k$
are $k$ distinct real zeroes of $\widetilde P$. Therefore from the  
 Fundamental Theorem of Algebra it follows that $\widetilde
 P(t)=c_0\prod_{j=1}^k(t-\cot\theta_j)$, and hence, in view of
 \eqref{eq:10},
 \begin{align*}
   g(\alpha)&= c_0(\sin\alpha)^k \prod_{j=1}^k (\cot \alpha -\cot\theta_j)
 = c_0 \prod_{j=1}^k (\cos \alpha -\cot\theta_j\sin\alpha)\\
 &= c_0 \prod_{j=1}^k \frac1{\sin\theta_j}(\sin\theta_j\cos \alpha -\cos\theta_j\sin\alpha)
 = c_0 \prod_{j=1}^k \frac1{\sin\theta_j}\sin(\theta_j-\alpha),
\end{align*}
 for all $\alpha\in\R\setminus\{\ell\pi:\,\ell\in\Z\}$. Then, by
 continuity, we conclude that
 \begin{equation}\label{eq:14}
   g(\alpha)=c_0 \prod_{j=1}^k
   \frac1{\sin\theta_j}\sin(\theta_j-\alpha),\quad\text{for all }\alpha\in\R.
 \end{equation}
We notice that \eqref{eq:14} implies that the values 
 $\theta_1,\theta_2,\dots,\theta_k$ are the unique  zeros of $g$ in
 the interval $(0,\pi)$. In particular, for every $j\in\{1,\dots,k\}$, 
 \begin{equation}\label{eq:15}
\theta_j \text{ is the unique zero
 of $g$ in the interval }\bigg( (j-1)\frac\pi k, j\frac\pi k\bigg).
\end{equation}
From \eqref{eq:5persym} and \eqref{eq:18} we have that 
\begin{align*}
&g\Big(\theta_1+(j-1)\frac\pi k\Big)\\
&=
\begin{cases}
  g(\theta_1),&\text{if }j\text{ is odd},\\
  -g\big(\theta_1+\pi+(j-1)\frac\pi
  k\big)=-g\big(\theta_1+(j-1+k)\frac\pi k\big)
  =-g(\theta_1),&\text{if }j\text{ is even},
\end{cases}\\
&=0,
\end{align*}
and hence, in view of \eqref{eq:15} and since $\theta_1+(j-1)\frac\pi
k\in \big( (j-1)\frac\pi k, j\frac\pi k\big)$, we have that 
\begin{equation}\label{eq:19}
\theta_j=\theta_1+(j-1)\frac\pi
k,\quad\text{for all }j=1,\dots,k.
\end{equation}
From \eqref{eq:5persym} it follows that
$g\big(-\theta_1+\frac{2\pi}{k}\big)=g(-\theta_1)=g(2\pi -\theta_1)=g(\theta_1)=0$;
therefore, since 
$-\theta_1+\frac{2\pi}{k}\in \big(\frac\pi k,\frac{2\pi}k\big)$, from 
\eqref{eq:15} and \eqref{eq:19} we can conclude that 
$-\theta_1+\frac{2\pi}{k}=\theta_2=\theta_1+\frac\pi k$ and hence
$\theta_1=\frac\pi{2k}$. Then from \eqref{eq:19} we deduce that 
\begin{equation*}
\theta_j=\frac\pi{2k}(2j-1),\quad\text{for all }j=1,\dots,k,
\end{equation*}
thus reaching the conclusion in view of \eqref{eq:14}.
\end{proof}

\begin{Lemma}\label{l:coseno}
 Let $k$ be an odd natural number. Then
\[
\prod_{j=1}^k
 \sin\Big(\frac\pi{2k}(2j-1)-\alpha\Big)=2^{1-k}\cos(k\alpha)
\]
for all $\alpha\in\R$.
\end{Lemma}
\begin{proof}
Since the complex numbers $e^{i\frac{2\pi}{k}j}$ with $j=1,2,\dots,k$
are $k$-th distinct roots of unity and $k$ is odd, we have that 
\begin{equation}\label{eq:20}
  1-z^k=\prod_{j=1}^k\left(e^{i\frac{2\pi}{k}j}-z\right),\quad\text{for
  all }z\in\C.
\end{equation}
Since 
\begin{align*}
  \sin&\Big(\frac\pi{2k}(2j-1)-\alpha\Big)=
\frac{e^{i\left(\frac\pi{2k}(2j-1)-\alpha\right)}-e^{-i\left(\frac\pi{2k}(2j-1)-\alpha\right)}}{2i}\\
&=
\frac1{2i}e^{-i\alpha}e^{-i\frac{\pi}{2k}(2j+1)}
\left(e^{i\frac{2\pi}{k}j}-e^{i(2\alpha+\frac\pi k)}\right) =
\frac1{2}e^{-i\alpha}e^{-i\frac{\pi}{2k}(2j+1+k)}
\left(e^{i\frac{2\pi}{k}j}-e^{i(2\alpha+\frac\pi k)}\right),
\end{align*}
from \eqref{eq:20} we deduce that
\begin{align*}
  \prod_{j=1}^k
 \sin\Big(\frac\pi{2k}(2j-1)-\alpha\Big)&=\frac1{2^k}e^{-ik\alpha}e^{-i\frac{\pi}{k}\sum_{j=1}^kj}
e^{-i\frac{\pi}{2}(1+k)}
\prod_{j=1}^k\left(e^{i\frac{2\pi}{k}j}-e^{i(2\alpha+\frac\pi
    k)}\right)\\
&=\frac1{2^k}e^{-ik\alpha}
e^{-i\pi(1+k)}\left(1-e^{i(2k\alpha+\pi)}\right)\\
&=\frac1{2^k}e^{-ik\alpha}
\left(1+e^{2ki\alpha}\right)=\frac1{2^k}
\left(e^{-ik\alpha}+e^{ki\alpha}\right)=2^{1-k}\cos(k\alpha)
\end{align*}
thus proving the lemma.
\end{proof}

\begin{proof}[Proof of Theorem \ref{t:main}]
From Lemmas \ref{l:fattorizzazione} and \ref{l:coseno} it follows
that, under the assumptions of Lemma \ref{l:taylor} and \eqref{eq:54},
the polynomial 
$P$ in \eqref{eq:tay}-\eqref{eq:polinomio} 
is given by
\[
  P(r\cos\alpha,r\sin\alpha)=-4\frac{|\beta_2|^2}{\pi}{\mathfrak m}_k r^k\cos(k\alpha),
\]
thus proving the conclusion in the case in which assumption
\eqref{eq:54} is satisfied. The general case $\beta_1\neq0$ can be
easily reduced to the case $\beta_1=0$ by a change of the cartesian
coordinate system $(x_1,x_2)$ in $\R^2$ which rotates the axes in such
a way that the positive $x_1$-axis is tangent to one of the $k$ nodal
lines of $\varphi_0$ ending at $0$.  If $\beta_1\neq 0$ and $\alpha_0$
is defined in \eqref{eq:alpha0}, the nodal lines of $\varphi_0$ at $0$
have tangent half-lines forming with the $x_1$-axis angles of
$\alpha_0+\frac{2\pi}k j$, $j=0,1,\dots,k-1$. If $\tilde
\varphi_0(x)=\varphi_0(R(x))$ and $\tilde
\varphi_a(x)=\varphi_a(R(x))$ with
\[
R(x_1,x_2)=
 \begin{pmatrix}
   \cos\alpha_0&   -\sin\alpha_0\\[3pt]
   \sin\alpha_0&   \cos\alpha_0
 \end{pmatrix}
 \begin{pmatrix}
   x_1\\
x_2
 \end{pmatrix},
\]
it is easy to verify that $\tilde
\varphi_0,\tilde\varphi_a$ solve problems
\begin{equation*}
   (i\nabla + A_0)^2 \tilde\varphi_0 = \lambda_0 \tilde\varphi_0,
   \quad 
 (i\nabla + A_{R^{-1}(a)})^2 \tilde\varphi_a = \lambda_a \tilde\varphi_a,
\end{equation*}
in the domain $R^{-1}(\Omega)$. Moreover 
\begin{equation*}
  r^{-k/2} \tilde\varphi_0(r(\cos t,\sin t)) \to 
  \tilde\beta_1
  \frac{e^{i\frac t2}}{\sqrt{\pi}}\cos\Big(\frac k2
  t\Big)+\tilde\beta_2 
  \frac{e^{i\frac t2}}{\sqrt{\pi}}\sin\Big(\frac k2
  t\Big) \quad \text{in }C^{1,\tau}([0,2\pi],\C)
\end{equation*}
as $r\to0^+$, where
\[
\begin{pmatrix}\tilde\beta_1\\\tilde\beta_2
\end{pmatrix}
=e^{i\frac\alpha2}
 \begin{pmatrix}
   \cos(\frac k2 \alpha_0)&   -\sin(\frac k2\alpha_0)\\[3pt]
   \sin(\frac k2 \alpha_0)&   \cos(\frac k2\alpha_0)
 \end{pmatrix}
\begin{pmatrix}\beta_1\\\beta_2
\end{pmatrix}
.\] 
From
\eqref{eq:alpha0} it follows that $\tilde \beta_1=0$ and hence
$|\tilde\beta_2|^2=|\beta_1|^2+|\beta_2|^2$.
Since we have already proved the theorem in the case $\beta_1=0$, we
know that 
\[
\frac{\lambda_0-\lambda_a}{|a|^k}\to 
-4\frac{|\tilde \beta_2|^2}{\pi}{\mathfrak m}_k
\cos(k\alpha),
 \quad \text{as $a\to0$ with $R^{-1}(a)=|a|(\cos\alpha,\sin\alpha)$,}
\]
which yields
\[
\frac{\lambda_0-\lambda_a}{|a|^k}\to 
-4\frac{|\beta_1|^2+|\beta_2|^2}{\pi}{\mathfrak m}_k
\cos(k(\theta-\alpha_0)),
 \quad \text{as $a\to0$ with $a=|a|(\cos\theta,\sin\theta)$},
\]
thus concluding the proof.
\end{proof}

\end{document}